\documentclass[12pt,reqno,tbtags]{amsart}
\usepackage{amssymb}
\usepackage{color}
\usepackage{natbib}
\bibpunct[, ]{[}{]}{;}{n}{,}{,}

\def\ignore#1{}
\ignore{
\usepackage{fancyhdr}
\lhead{}
\rhead{}
\fancyhead[CE]{\sc Barbour, Brightwell and Luczak}
\fancyhead[CO]{\sc Concentration and cut-off}

\pagestyle{fancy}
}

\definecolor{mygreen}{RGB}{0,160,50}

\numberwithin{equation}{section}

\makeatletter
\newtheorem*{rep@theorem}{\rep@title}
\newcommand{\newreptheorem}[2]{%
\newenvironment{rep#1}[1]{%
 \def\rep@title{#2 \ref{##1}}%
 \begin{rep@theorem}}%
 {\end{rep@theorem}}}
\makeatother

\DeclareMathOperator*{\esssup}{ess\,sup}

\newtheorem{theorem}{Theorem}[section]
\newreptheorem{theorem}{Theorem}
\newtheorem{lemma}[theorem]{Lemma}
\newtheorem{proposition}[theorem]{Proposition}

\theoremstyle{definition}

\theoremstyle{remark}

\newcounter{thmenumerate}

\newcounter{xenumerate}

\newcommand{\m}[1]{\marginpar{\tiny{#1}}}

\newcommand\dev{\operatorname{\mathrm dev}}
\newcommand\var{\operatorname{\mathrm var}}

\newcommand\E{\operatorname{\mathbb E{}}}

\renewcommand\Pr{\operatorname{\mathbb P{}}}

\newcommand\Bi{\operatorname{Bi}}

\newcommand\eps{\varepsilon}

\renewcommand\phi{\varphi}
\newcommand\g{\gamma}

\def\ui{^{(1)}}
\def\ut{^{(2)}}
\def\uk{^{(k)}}

\newcommand\cA{\mathcal A}

\newcommand\cF{\mathcal F}
\newcommand\cG{\mathcal G}

\newcommand\cJ{\mathcal J}

\newcommand\cL{\mathcal L}

\newcommand\cN{\mathcal N}

\newcommand\wX{\widehat{X}}

\newcommand\wS{\widehat{S}}
\newcommand\wP{\widehat{P}}
\newcommand\wQ{\widehat{Q}}
\newcommand\wa{\widehat{a}}

\newcommand\wbeta{\widehat{\beta}}

\newcommand\weta{\widehat{\alpha}}

\renewcommand\P{{\mathbb P}}
\newcommand\px{p_{\rm exit}}

\newcommand\N{{\mathbb N}}
\newcommand\Z{{\mathbb Z}}

\newcommand\R{{\mathbb R}}

\def\Def{\ :=\ }

\def\tp{{\widetilde p}}

\newenvironment{proofof}[1]{\noindent {\bf
Proof of #1}.}{\hfill $\square$\par\smallskip\par}

\newcommand{\eqs}{\begin{eqnarray*}}
\newcommand{\ens}{\end{eqnarray*}}

\newcommand{\eq}{\begin{equation}}
\newcommand{\en}{\end{equation}}

\newcommand{\eqa}{\begin{eqnarray}}

\newcommand{\ena}{\end{eqnarray}}

\def\ex{{\mathbb E}}
\def\t{\tau}

\def\a{\alpha}
\def\th{\theta}
\def\Le{\ \le\ }

\def\HG{{\rm HG\,}}
\def\law{{\mathcal L}}
\def\tY{{\widetilde Y}}

\def\L{\Lambda}
\def\d{\delta}
\def\e{\varepsilon}
\def\Ref#1{(\ref{#1})}
\def\D{\Delta}

\def\Blb{\left\{}
\def\Brb{\right\}}

\def\dtv{d_{TV}}
\def\quarter{\tfrac14}
\def\half{\tfrac12}

\def\pr{{\mathbb P}}
\def\un{^{(n)}}
\def\giv{\,|\,}
\def\Eq{\ =\ }
\def\r{\rho}
\def\b{\beta}
\def\a{\alpha}
\def\g{\gamma}
\def\d{\delta}
\def\m{\mu}
\def\n{\nu}

\def\non{\nonumber}
\def\bx{{\mathbf x}}
\def\bu{{\mathbf u}}
\def\bv{{\mathbf v}}
\def\bc{{\mathbf c}}
\def\bm{{\mathbf m}}

\def\by{{\mathbf y}}
\def\bz{{\mathbf z}}
\def\bw{{\mathbf w}}

\def\bb{{\mathbf b}}

\def\bJ{{\mathbf J}}

\def\wxn{\wX\un}

\def\tipi{{\widetilde\pi}}
\def\ps{\psi}

\def\hcA{{\widehat\cA}}
\def\tr{{\rm Tr\,}}

\def\z{\zeta}
\def\wS{{\widehat S}}

\def\s{\sigma}
\def\wZ{{\widehat Z}}
\def\FF{{\mathcal F}}
\def\SSS{E}
\def\wtS{{\widetilde S}}

\def\adb{}
\def\grb{} 
\def\adbn{} 
\def\adbng{} 
\def\grbb{} 
\def\grbc{} 
\def\grbd{} 

\begin{document}
\title[Concentration and cut-off]
{Long-term concentration of measure and cut-off}


\author{A. D. Barbour} \thanks{ADB: Work
supported in part by Australian Research Council Grants Nos DP150101459 and DP150103588, and by the ARC Centre of Excellence for
Mathematical and Statistical Frontiers, CE140100049.  Thanks to the mathematics departments of the University of Melbourne and Monash University for their kind hospitality.}
\address{Institut f\"ur Mathematik, Universit\"at Z\"urich, Winterthurertrasse 190, CH-8057 Z\"urich}
\email{a.d.barbour@math.uzh.ch}

\author{Graham Brightwell}
\address{Department of Mathematics, LSE}
\email{g.r.brightwell@lse.ac.uk}

\author{Malwina Luczak} \thanks{MJL: Research partly supported by an EPSRC Leadership Fellowship EP/J004022/2 and partly by ARC Future Fellowship FT170100409.}
\address{School of Mathematics and Statistics, University of Melbourne}
\email{mluczak@unimelb.edu.au}


\begin{abstract}
We present new concentration of measure inequalities for Markov chains, generalising results for chains that are contracting in Wasserstein distance.
These are particularly suited to establishing the cut-off phenomenon for suitable chains.
We apply our discrete-time inequality to the well-studied Bernoulli-Laplace model of diffusion,
and give a probabilistic proof of cut-off, recovering and improving the bounds of Diaconis and Shahshahani.
We also extend the notion of cut-off to chains with an infinite state space, and illustrate this in a second example,
of a two-host model of disease in continuous time.  We give a third example, giving concentration results for the supermarket model, illustrating
the full generality and power of our results.
\end{abstract}

\maketitle

\section{Introduction}\label{S:intro}

We have two main aims in this paper.  The first is to develop some new concentration of measure inequalities for Markov chains,
both in discrete and continuous time, and the second is to introduce a wider perspective on the cut-off phenomenon for convergence to equilibrium of Markov chains.
Our past work suggests a strong connection between long-term concentration of measure, rapid mixing, and cut-off: this paper is an attempt to formalise, explain
and illustrate this.

Our concentration of measure inequalities generalise and extend earlier results applicable for chains {\em contracting in Wasserstein distance}, which
means that there is a metric on the state space so that the chain makes only short steps with respect to the metric, and a coupling of two copies of
the chain so that the distance between the two copies decreases in expectation -- in the language of Ollivier~\cite{Ollivier}, this means that the chain has
{\em positive coarse Ricci curvature}.  For discrete-time Markov chains with positive coarse Ricci curvature, Ollivier proves that any real-valued function
of the Markov chain that is Lipschitz with respect to the metric remains well-concentrated around its expectation for all time, and in equilibrium; a
similar result follows from results of Luczak~\cite{l08} proved independently at around the same time.  Paulin~\cite{Paulin} gives a more general framework,
obtaining concentration results, and bounds on the mixing time, in cases where the ``multi-step coarse Ricci curvature'' is positive, even if the coarse
Ricci curvature is not.  The concentration results proved in these papers, as well as in the present paper, are of the ``Gaussian then exponential'' type, akin to
Bernstein's Inequalities:
the probability of deviations of at least $m$ from the mean is of order $e^{-cm^2}$ for small $m$ and $e^{-cm}$ for large $m$ -- Ollivier gives
examples where this is the best possible form of the concentration inequality.

Our new results in discrete time do not rely on the existence of a well-behaved metric on the state space, and require only conditions regarding the
function of interest.  Thus we obtain stronger concentration results for functions of the chain that evolve much more slowly
than the total transition rate of the chain, as long as they are contractive, in a suitable sense.  We recover essentially the same result as Ollivier
in the case of positive coarse Ricci curvature, and we can also obtain results very similar to those of Paulin, but our results can also be used to prove concentration
of measure in other settings.  The application we give in the final section of our paper gives a concentration result that we do not
know how to obtain by other means.

We also give analogous concentration inequalities for continuous-time Markov chains.  These are entirely new, although, for chains contracting
in Wasserstein distance, similar results could be obtained via the methods and results of Ollivier~\cite{Ollivier}, Luczak~\cite{l08} and Paulin~\cite{Paulin}.  Veysseire~\cite{Veysseire} gives definitions and results for coarse Ricci curvature in continuous time, but does not prove
any results that are closely related to ours.

We now turn to the cut-off phenomenon.  For a Markov chain
$(X(t))$, with initial state $X(0)=x$, consider the total variation distance between the law of the process at time~$t$ and the equilibrium
distribution.  The chain is said to exhibit the cut-off phenomenon if this distance falls from near~1 to near~0 over a window of time that is much
shorter than the mixing time.  In previous work, it is assumed that the state space is finite, and the starting state $x$ is chosen to maximise the
mixing time.  We present a version of the definition allowing for an infinite state space, and for variation of the mixing time over a region of
potential initial states, with a cut-off window of width that is uniform across this region.

Our concentration of measure inequalities, combined with coupling arguments, are well-suited to proving cut-off, and we illustrate this with two
examples of independent interest.  The first is the well-known Bernoulli-Laplace model of diffusion: there are initially $n$ red balls in one urn and
$n$ black balls in another, and at each time step one ball from each urn is chosen uniformly at random and the two balls are exchanged.  Cut-off
was proved for this model in 1987 by Diaconis and Shahshahani~\cite{ds87} using algebraic techniques: we provide a probabilistic proof, essentially recovering the
bound of Diaconis and Shahshahani for the upper tail of the distribution of the mixing time, while providing a sharper bound for the lower tail.

Our second application concerns a continuous-time model of a disease with two types of host, each infecting the other;
the disease is supported at a low level in a population by immigration of both types of infected host from outside.  This example illustrates both
the application of our new continuous-time concentration inequality and our new concept of cut-off, as the state space is infinite and the mixing
time varies significantly depending on the initial conditions.

In both of the sample applications above, the chain we examine is contractive in Wasserstein distance, and variants of the results we obtain
could also be obtained from concentration inequalities in earlier work.  We also present a third application which uses the full power of our new
continuous-time inequality; this treats the supermarket model, a well-known queueing system, with a certain range of parameter values.
In this example, we utilise facts about the equilibrium distribution from a paper of Brightwell, Fairthorne and Luczak~\cite{BFL}, alongside our long-term
concentration result, to show tight concentration in equilibrium of the number of empty queues.

\subsection{Concentration of measure inequalities}

Our general concentration inequality for discrete-time Markov chains appears as Theorem~\ref{thm.concb-general}, and the special case where the
chain is contracting in Wasserstein distance with respect to a suitable metric as Corollary~\ref{ADB-contract-disc-lem}.  Part~(a) of
Theorem~\ref{ADB-contract-disc-lem} is very similar to Theorem~32 of Ollivier~\cite{Ollivier} -- that result is for the equilibrium distribution of
the chain, whereas ours is for finite-time distributions, but Ollivier's Remark~34 indicates that the proof in his paper transfers to the finite-time
case.  A similar result for chains contracting in Wasserstein distance also follows readily from Theorem~4.5 of Luczak~\cite{l08}.  We give more details
after we have given precise definitions and statements of theorems.

There is another quite different recent strand of work providing tools to show concentration of measure and rapid mixing for a given function of a Markov chain, useful in
circumstances where the function mixes more rapidly than the chain itself.  See Watanabe and Hayashi~\cite{WH} and Rabinovitch, Ramdas, Jordan and Wainwright~\cite{rrjw}.

Results similar to Theorem~\ref{thm.concb-general} appear in earlier works of the third author, some unpublished, and a number of other
applications are to be found in these papers, as well as in Gheissari, Lubetzky and Peres~\cite{GLP}.
The flavour of the inequality is similar to that of Luczak~\cite{l08}, but Theorem~\ref{thm.concb-general} can be much more powerful when the chain makes
frequent transitions that do not alter the value of the function of interest.

One example where this is relevant is the supermarket model, as studied
in the final section of this paper, where the number of queues of length~$k$ only changes infrequently for some values of~$k$.

Another example is the alternative routing model of Gibbens, Hunt and Kelly~\cite{GHK}.  Here, there are links of limited capacity between each pair of
nodes in a phone network; requests for pairs of nodes to be connected arrive according to a Poisson process, and these can be met either by using the
direct link or by using some path of two links.  Different protocols have been proposed and studied for choosing the route; one such is to use the
direct link if it has spare capacity, and if not then to inspect $d\ge 1$ links of two routes, and use one of those with most spare capacity.  In an
unpublished preprint of Luczak~\cite{L2012}, an earlier version of Theorem~\ref{thm.concb-general} is used to prove a differential equation approximation for this model, extending earlier results of Crametz and Hunt~\cite{ch} and Graham and M\'el\'eard~\cite{gm}.  The equilibrium behaviour of the model
is studied via a similar approach in an unpublished preprint of Brightwell and Luczak~\cite{BL2013}.
The same methods can be used to treat other routing protocols.  The key principle is that quantities such as the number of occupied links incident with
a given node change far less often than the overall state of the network.
Our new result, Theorem~\ref{thm.concb-general}, improves on the earlier version in Luczak~\cite{L2012} (Theorem~2.3) by weakening and simplifying its hypotheses.

The corresponding inequality for continuous-time Markov chains is Theorem~\ref{thm.concb-continuous}, and the special case for chains that are
contracting in Wasserstein distance is Theorem~\ref{ADB-contract-cts-lem}.  Our proof for continuous time uses different methods to those used
for discrete time (although both proofs draw on principles of concentration of measure for martingales), and it is perhaps a little surprising that the resulting theorems are nearly exact analogues of each other.  In Brightwell and Luczak~\cite{BL2013},
a continuous-time model is analysed (somewhat awkwardly) by applying discrete-time concentration of measure inequalities from~\cite{L2012} to its
jump chain; it seems that this analysis would be eased by direct application of our new continuous-time inequalities, and we plan to produce an
improved version of~\cite{BL2013} in the future.

Our notion of contraction in Wasserstein distance is very different in flavour from that of contraction in total variation distance, as studied by
Marton~\cite{Marton96} and others subsequently.  In particular, for a chain to exhibit contraction in total variation distance, it is necessary that,
from any two states, there is a positive probability that two coupled chains started in these states coalesce in a single step.

\subsection{Cut-off}

We now discuss the cut-off phenomenon in the convergence to equilibrium for sequences~$X\un$ of \grb {Markov} chains.

Let $\law_{x}(X\un (t))$ denote the distribution of~$X\un$ when $X\un (0)=x$, and let $\pi\un$ be the equilibrium distribution of $X\un$.
Let \grbc{$S\un$} denote the state space of the chain $X\un$.

In earlier papers (for instance, Diaconis and Shahshahani~\cite{ds87} and Levin, Luczak and Peres~\cite{llp}), cut-off is defined as follows, in the
case where the state space $S\un$ is finite for each $n$.  The {\em worst-case distance to stationarity} for the chain $X\un$ at time $t$ is
$$
d_n(t) = \max_{x\in \grbc{S}\un} \dtv \Bigl(\law_x\bigl( X\un(t) \bigr) , \pi^{(n)} \Bigr),
$$
and the sequence $X\un$ of chains is said to exhibit cut-off at time $t_n$ with {\em window width} $w_n$ if $w_n = o(t_n)$ and
$$
\lim_{s \to \infty} \liminf_{n \to \infty} d_n(t_n - sw_n) =1; \quad \lim_{s \to \infty} \limsup_{n \to \infty} d_n(t_n + sw_n) =0.
$$
In other words, for a large constant $s$, at time $t_n + s w_n$, the chain $X\un$ is nearly in equilibrium, whatever the starting state; on the other hand,
there is a starting state $x \in \grbc{S}\un$ such that the chain $X\un$ starting from state $x$ is very far from equilibrium at time $t_n - sw_n$.

In many cases where cut-off, with window width $w_n$, can be proven, the situation is typically as follows, with a proof involving two separate arguments.
The state space has a metric, and the Markov chain makes jumps that are small with respect to this metric.
The equilibrium distribution is concentrated around some point $y$ (suitably scaled with $n$) in the state space.  If the chain is started at some
``distant'' point $x$, one shows that its trajectory is concentrated around its expectation, up until some time $t_n(x)$ when the expectation becomes
suitably close to $y$.  Once in the neighbourhood of $y$, one seeks a coupling with a copy of the chain
in equilibrium, where coalescence takes place in time of order $w_n$.  One example of such a proof was given by Levin, Luczak and Peres~\cite{llp}, and
our examples in Sections~\ref{DS} and~\ref{parasites} both illustrate this general approach.

Similar behaviour is often to be found in examples where the state space is infinite, and there is no ``most distant'' starting point from equilibrium.
For instance, in a population model, there may be no effective upper bound on the initial size of a population.  Thus we find it useful to introduce a
more general notion of cut-off, where the mixing time $t_n(x)$ depends on the initial state, but the window width $w_n$ is independent of the starting
state.  The proof scheme above can then be applied, provided we restrict the class of allowed initial states to exclude (a)~states $x$ too close to the
point $y$ around which the equilibrium is concentrated, where the ``travel time'' $t_n(x)$ from $x$ to $y$ will be of similar or smaller order to the
time $w_n$ required for coalescence of the coupled chains in the neighbourhood of $y$, and (b)~possibly also states $x$ extremely distant from
$y$, where the fluctuation in the travel time exceeds the window width $w_n$.

We now give our formal definition of cut-off, which extends the previous definition, and in particular allows for an infinite state space.
For~$\SSS_n$ a subset of the state space~$S\un$ of~$X\un$,
let $(t_n(x),\,x\in \SSS_n)$ be a collection of non-random times, and let $(w_n)$ be a sequence of numbers such
that $\lim_{n\to\infty}\inf_{x \in \SSS_n} t_n(x)/w_n = \infty$.    We say that~$X\un$ exhibits
{\em cut-off} at time $t_n(x)$ on~$\SSS_n$ with {\em window width}~$w_n$, if there exist (non-random)
constants~$(s(\e),\,\e>0)$ such that, for any $\e > 0$ and for all~$n$ large enough,
\eqa
   &&\dtv\Bigl(\law_{x}\bigl(X\un (t_n(x) - s(\e)w_n)\bigr),\pi^{(n)}\Bigr) \ >\ 1-\e, \non\\
   &&\dtv\Bigl(\law_{x}\bigl(X\un(t_n(x) + s(\e)w_n)\bigr),\pi^{(n)}\Bigr) \ <\ \e,
       \label{ADB-cutoff}
\ena
uniformly for all $x \in \SSS_n$.

In some examples, the travel time $t_n(x)$ can be taken not to depend on $x$, as long as $x \in \SSS_n$.
We say that~$X\un$ exhibits cut-off at~$t_n$ on~$\SSS_n$ with window width~$w_n$, for a sequence $(t_n,\,n\ge1)$,
if the $t_n(x)$ in the definition above can be set equal to $t_n$ for all $n$ and all $x \in \SSS_n$.
An illustration of this last concept comes in Section~\ref{DS}; the idea here is that the expected ``travel times'' from all suitably
distant starting states are nearly equal.


Our concentration of measure results are suited to showing that a Markov chain closely follows an almost deterministic trajectory until it reaches the neighbourhood
near where the equilibrium is concentrated.  In order to complete a proof of cut-off, one needs to show that convergence to equilibrium is rapid once
that neighbourhood has been reached.  Proposition~\ref{hitting} gives conditions guaranteeing that a Markov chain taking non-negative real values,
with a non-positive drift in all positive states, reaches~0 quickly with high probability.  This implies an upper bound on the
coalescence time for the two copies of the chain in a contracting coupling. 
We give such a result only in continuous time, and apply it in our continuous-time sample application in Section~\ref{parasites}.
Our proof of Proposition~\ref{hitting} is based on the proof of a discrete-time analogue appearing as Proposition~17.19 of Levin, Peres
and Wilmer~\cite{LPW}.  Our application in Section~\ref{DS} requires a sharper coupling result specific to the model; using some version of
Proposition~17.19 from~\cite{LPW} would give weaker bounds on the tail of the distribution of the mixing time.

\subsection{Applications}

We give three examples.  The first two feature chains that are contracting in Wasserstein distance, illustrating both our methods and the cut-off
phenomenon.  In the third example, we prove results about concentration of the equilibrium distribution by using the full strength
of our new concentration inequalities.

Section~\ref{DS} concerns the Bernoulli--Laplace model of diffusion, originally investigated in the context of cut-off
by Diaconis and Shahshahani~\cite{ds87}.
In this discrete-time model, there are two urns each containing $n$ balls, with $n$ red and $n$ black balls in
total: at each time step, one ball is chosen uniformly at random from each urn and the two are exchanged.  The state of the system after $r$ steps
is captured by the number $X(r)$ of red balls in the left urn, and one compares the distribution of $X(r)$ with the stationary distribution (which is
concentrated around $n/2$).  Diaconis and Shahshahani prove cut-off for $X(r)$ at time $\frac14 n \log n$ with window width $n$.  Indeed, their proof
establishes cut-off not only for the most distant starting states (where $X(0) = 0$ or $n$) but on any set
$E_n(\eps) = \{ j : |j-\frac n2| \ge \eps n\}$.  They also give specific exponential rates for the tail of the distribution of the mixing time.
The methods used by Diaconis and Shahshahani are algebraic: we give an alternative proof, using our concentration of measure results.  Our proof gives the same
exponential rate for the upper tail as in~\cite{ds87}, although our proof does not give information about the extreme end of the tail, where the
total variation distance between the distribution at time $r$ and the equilibrium distribution is below $n^{-1/2} \log^2 n$.  Our methods yield a doubly
exponential rate for the lower tail, improving on the results of Diaconis and Shahshahani.

In Section~\ref{parasites}, we consider a toy model of a subcritical two-host infection,
maintained by immigration of infectives from outside, at rates that are constant multiples of a scale parameter~$n$.
Our model is appropriate in circumstances where the number of infectives is small compared to the total population size,
and the expected number of infectives of each \grbc{type of host} satisfies a linear equation with a fixed point $n \bc \in \R_+^2$.
We consider an arbitrary starting state $\bx$ within an annular region $E_n(\z) = \{ \by : n \z \le | \by - n \bc |  \le n/\z \}$,
where $\z \in (0,1)$, and we show cut-off at $t_n(\bx)$ with window width~1 over this region.  Here the travel time $t_n(\bx)$ is bounded
between two constants times $\log n$, but varies over the region $E_n(\z)$, for any $\z \in (0,1)$.

In Section~\ref{supermarket}, we consider the supermarket model.  In this $n$-server queueing model, customers arrive according to a Poisson process
at rate $\lambda n$, where $\lambda<1$, and inspect $d\ge 1$ queues before joining a shortest queue among these~$d$.  The service time of each customer
is exponential of mean~1.  We consider a parameter regime where $\lambda$ tends to~1 as $1 - n^{-\alpha}$, and $d$ grows as $n^\beta$, where $\alpha$ and
$\beta$ are constants satisfying certain inequalities.  We choose the precise parameter range so that, as shown by Brightwell, Fairthorne and
Luczak~\cite{BFL}, the maximum queue length in equilibrium is~2 with high probability, and most queues have length exactly~2.  For this model, we study the
distribution of the number of empty queues, and show that it is concentrated within order $n^{\frac12(1-\beta)}$ of its mean $n^{1-\alpha}$.
The application is chosen to illustrate the power of our general results; most transitions of the chain do not affect the number of empty queues, so that our
methods give stronger concentration results than we are able to obtain by any other means.  The techniques we use will extend readily to other parameter ranges.

Further consequences of inequalities Theorem~\ref{thm.concb-general} and Theorem~\ref{thm.concb-continuous}
will be explored in future work.


\section{Concentration inequalities: discrete time}
\label{S:conc}



In this section, we first state and prove a general concentration of measure inequality designed for the analysis of discrete-time Markov chains,
generalizing results of Luczak~\cite{l08}.  We then show how to recover a version of a result of Ollivier~\cite{Ollivier} for
contracting chains, which is perhaps more appealing and still fairly widely applicable.  Next, we outline how to use the inequality when
we have a coupling of two copies of the chain which is ``approximately contracting'' in the function of interest.  Finally, we give a toy example to
illustrate the application of the inequalities.

\subsection{Main result}

Here and throughout, we use $\Z_+$ to denote the non-negative integers.
Let $X=(X(i))_{i \in \Z_+}$ be a discrete-time Markov chain with a discrete state space~$\grbb{S}$ and
transition probabilities $P(x,y)$ for $x,y \in S$.
We allow $X$ to be lazy; that is, we allow $P(x,x) > 0$ for $x \in S$.

For $x \in S$, we set
\eqa
   N(x) &:=& 
   \{y \in S: P(x,y) > 0\}.\non
\ena
For $k \in \Z_+$ and a function $f\colon S \to \R$, define the function $P^k f$ by
$$
     (P^k f)(x) \Def \E_x [f(X(k))],  \quad x \in S,
$$
whenever it exists,
where $\E_x$ and~$\P_x$ are used to denote conditional expectation and probability given $X(0) = x$.

\begin{theorem}
\label{thm.concb-general}
Let $P$ be the transition matrix of a discrete-time Markov chain $(X(i))_{i \in \Z_+}$ with
discrete state space $S$.
Let $\wtS$ be a subset of $S$.
Let $f\colon S \to \R$ be a function such that $(P^i f)(x)$ exists for all $x\in S$
and $i\in \Z_+$, and satisfying, for all $i\in\Z_+$,
\eqa
    \big|(P^i f)(x) - (P^i f)(y)]\big| \Le \beta,\quad
                x \in \wtS ,\ y \in N(x); \label{cond-gen-5}\\
    \sum_{y \in N(x)} P(x,y) \bigl((P^i f)(x) - (P^i f)(y)\bigr)^2 \Le \alpha_i, \quad
                x \in \wtS,  \label{cond-gen-4}
\ena
where $\beta$ and $(\alpha_i)_{i \in \Z_+}$ are  positive constants.
 Set $a_k := \sum_{i=0}^{k-1} \alpha_i$, $k\ge1$.
Define $A_k := \{X(i) \in \wtS \mbox{ for } 0 \le i \le k-1\}$,
the event that $(X(i))$ stays in $\wtS$ for the first $k-1$ steps.
Then, for all $x_0 \in \wtS$ and all $m \ge 0$,
$$
 \Pr_{x_0} \Bigl( \bigl \{ \bigl|f(X(k))-(P^k f)(x_0) \bigr|\ge m \bigr \} \cap A_k \Bigr)
   \Le 2e^{-m^2/(2a_k + 4\beta m/3)}.
$$
\end{theorem}

The conditions of the theorem are what is needed to fit into the framework of bounded differences (Bernstein-like) inequalities, and the expression in the assumption on~$f$
is, as we shall see, exactly what emerges when we bound conditional variances.

Evidently $a_k$ increases with~$k$.  Under a contractivity assumption, as we shall see shortly, the $\alpha_i$ can be taken to tend to~0 exponentially, so that the
$a_k$ are uniformly bounded: this means that we have a concentration of measure bound that is uniform in~$k$.  The result can also be applied in circumstances where the
$\alpha_i$ either converge more slowly to~0, or increase not too rapidly: in these cases, we obtain tighter concentration of $f(X(k))$ for smaller values of~$k$.

Theorem~\ref{thm.concb-general} improves on Theorem~4.5 of Luczak~\cite{l08} by using~\Ref{cond-gen-4}
to define~$\a_i$, instead of the cruder bound
\[
    L^2 \sum_{y \in N(x)} P(x,y) W\bigl(\law_x(X(i)),\law_y(X(i))\bigr)^2,
\]
where $f$ is assumed to be a Lipschitz function with Lipschitz constant~$L$, and~$W$ denotes the Wasserstein distance (both defined
with respect to the same metric on the state space $S$).
This is particularly important in contexts in which~$f(X(i))$ evolves significantly more
slowly than~$X(i)$ itself, because many of the transitions of~$X$ do not change the value of~$f$.
An example where this is relevant is the supermarket model, discussed in the final section of this paper, as well as the alternative routing model
of Gibbens, Hunt and Kelly~\cite{GHK} and its generalisation, as studied in Brightwell and Luczak~\cite{BL2013}.  (These particular examples are set up as continuous-time
Markov chains, for which our companion inequality,
Theorem~\ref{thm.concb-continuous}, is more naturally applicable, though it is also natural to consider their discrete-time analogues.)
Theorem~\ref{thm.concb-general} also improves on Theorem~2.3 of Luczak~\cite{L2012}, by weakening and
simplifying its hypotheses.

In the case where the hypotheses of Theorem~\ref{thm.concb-general} are satisfied with $\wtS = S$, we can immediately derive a bound on the variance of
$f(X(k))$, valid for any fixed starting state $x_0$.  Indeed, we have
\begin{eqnarray}
\var(f(X(k)) &=& \int_{r=0}^\infty \Pr_{x_0} \Big( \big(f(X(k)) - (P^k f)(x_0) \big)^2 \ge r \Big) \, dr \nonumber \\
&\le& 2 \int_{r=0}^\infty \exp \Big( \frac{-r}{2a_k + 4\beta\sqrt r/3} \Big) \, dr  \nonumber\\
&\le& 2 \int_{r=0}^\infty e^{-r/4a_k} + e^{-3\sqrt r/8\beta} \, dr \nonumber \\
&=& 2 (4a_k + 128\beta^2/9) \le 8a_k + 29 \beta^2. \label{variance}
\end{eqnarray}


\subsection{Proof of Theorem~\ref{thm.concb-general}}

To prove Theorem~\ref{thm.concb-general}, we use a slight extension of a result of McDiarmid~\cite{cmcd98}.
Inequality~\eqref{ineq-cmcd-1} in Lemma~\ref{thm.mart-b} below is a `two-sided' version of inequality~(3.28)
in Theorem 3.15 of McDiarmid~\cite{cmcd98}; inequality~\eqref{ineq-cmcd-2} is a slight extension of inequality~(3.29) of McDiarmid~\cite{cmcd98},
in that we work with a non-deterministic bound on $|Z_i - Z_{i-1}|$, and is also two-sided.

For a square integrable random variable $Y$ and
a $\sigma$-field $\cG \subseteq \cF$, we use $\var (Y \mid \cG)$ to denote
the {\it conditional variance} of $Y$ on $\cG$.
\ignore{
is defined to be the $\cG$-measurable random variable given by
$$
    \var (Y \mid \cG) \Def \E \Big (\big(Y - \E (Y \mid \cG)\big)^2 \mid \cG \Big ).
$$
}


\begin{lemma}
\label{thm.mart-b}
Let $(\Omega, \cF, \P)$ be a probability space equipped with a filtration
$\{\emptyset, \Omega \} = \cF_0 \subseteq \cF_1 \subseteq \cdots \subseteq \cF_k$
in $\cF$. \ignore{, where $\cF_k$ is finite.}
Let $Z$ be an $\cF_{k}$-measurable random variable with $\E Z = \mu$,
and let $Z_i = \E (Z \mid \cF_i)$, for $i=0, \dots, k$.
Let $\gamma$ and $\delta$ be constants such that $\sum_{i=1}^k \var(Z_i \mid \cF_{i-1})(\omega) \le \delta$ a.s.\
and $|Z_i(\omega) - Z_{i-1}(\omega)| \le \gamma$ a.s.\ for all $i=1, \dots, k$.
Then for any $m \ge 0$,
\begin{equation}
\label{ineq-cmcd-1}
  \P \big(\big|Z-\mu\big|\ge m\big) \Le 2e^{-m^2/(2 \delta+ 2\gamma m/3 )}.
\end{equation}

More generally, the following holds. For $\delta, \gamma \ge 0$, let
$$
   A(\delta, \gamma) \Def \Bigl\{\sum_{i=1}^k \var(Z_i \mid \cF_{i-1}) \le \delta\Bigr\}
     \cap \bigl\{\big|Z_i-Z_{i-1}\big| \le \gamma,\, 1\le i \le k\bigr\}.
$$
For any $m \ge 0$ and any values $\delta, \gamma \ge 0$,
\begin{eqnarray}
  \P \Bigl(\big\{\big|Z-\mu \big|\ge m \big\} \cap A (\delta, \gamma)
\Bigr)
        \Le 2e^{-m^2/(2\delta+ 2\gamma m/3)}. \label{ineq-cmcd-2}  
\end{eqnarray}
\end{lemma}

The proof is that of Theorem 3.15 (inequalities~(3.28) and (3.29)) in McDiarmid~\cite{cmcd98}, except that we use the indicator of the
event~$A(\delta,\gamma)$ instead of the event  $\Bigl\{\sum_{i=1}^k \var(Z_i \mid \cF_{i-1}) \le \delta \Bigr\}$.  The proof is rather like a
stopping argument, avoiding some technicalities.

\begin{proof}
Following McDiarmid~\cite{cmcd98}, we use Lemma 3.16~\cite{cmcd98}, which is as follows.  If $(Y_i)$ is a martingale difference sequence with
respect to a filtration $(\cF_i)$, where each $Y_i$ is bounded above, if~$I$ is an indicator random variable,
and if~$h$ is a real number, then
$$
   \E \Bigl(I e^{h \sum_{i=1}^k Y_i} \mid \cF_0 \Bigr)
          \Le \esssup \Bigl(I \prod_{i=1}^k \E (e^{hY_i} \mid \cF_{i-1}) \mid \cF_0 \Bigr).
$$
(The statement in~\cite{cmcd98} involves the supremum instead of the essential supremum: the notionally stronger
version is obtained by changing the $Y_i$ on a set of measure~0.  The proof is fairly straightforward by induction over a single-step
inequality.)

Now, for any random variable $X$ such that $X \le b$ and $\E X=0$,
we have $\E (e^X) \le e^{g(b) \var X}$, where $g(x) := (e^x-1-x)/x^2$ (see Lemma 2.8 in McDiarmid~\cite{cmcd98}).
So, for any $h$, defining the (possibly infinite) $\cF_{i-1}$ random variables $\var_i := \var (Z_i \mid \cF_{i-1})$
and $\dev_i^+ := \esssup (Z_i - Z_{i-1}\mid \cF_{i-1})$, we have
$$
   \E (e^{h(Z_i - Z_{i-1})} \mid \cF_{i-1} ) \Le e^{h^2 g(h \dev_i^+) \var_i}.
$$

Let $I$ be the indicator of the event $A(\delta, \gamma)$. It then follows that
\begin{eqnarray*}
   \E (I e^{h (Z - \mu)} ) & \le & \esssup \Bigl(I \prod_{i=1}^k e^{h^2 g(h \dev_i^+) \var_i}\Bigr) \\
      &\le& e^{h^2 \esssup \bigl(I \sum_{i=1}^k g(h \dev_i^+) \var_i\bigr)} \Le  e^{h^2 g(h \gamma) \delta}.
\end{eqnarray*}
Hence
\begin{eqnarray*}
\P ( \{Z- \mu \ge m \} \cap A(\delta,\gamma) ) & = & \P (I e^{h(Z-\mu)} \ge e^{hm} ) \\
 &\le& e^{-hm} \E (I e^{h(Z-\mu)}) \Le e^{-hm + h^2 g(h \gamma) \delta}.
\end{eqnarray*}
Optimising in~$h$, we set $h = \frac{1}{\gamma} \log(1 + \frac{m \gamma}{\delta})$ and use the
inequality $(1+x) \log (1+x) -x \ge x^2/(2+2x/3)$ for $x \ge 0$, as in the proof of Theorem~2.7 in McDiarmid~\cite{cmcd98}.

We obtain that
$$
\P ( \{Z- \mu \ge m \} \cap A(\delta,\gamma) ) \le e^{-m^2/(2\delta+ 2\gamma m/3)}.
$$
The same proof gives the same upper bound on $\P ( \{Z- \mu \le - m \} \cap A(\delta,\gamma) )$, and the result follows.
\end{proof}

\begin{proofof}{Theorem~\ref{thm.concb-general}}
We start 
by assuming that $\wtS = S$.
Let $({\mathcal F}_i)$ denote the natural filtration of $(X(i))_{i \in \Z_+}$.
We fix a function $f\colon S \rightarrow \R$, a natural number $k$, and an initial state $x_0 \in S$.
We consider the evolution of $(X(i))_{i \in \Z_+}$ for $k$ steps, conditional on $X(0)=x_0$.
%
Define the random variable $Z := f(X(k))$. Then,
for $i =0, \ldots, k$, $Z_i$ is given by
$$
   Z_i \Eq \E_{x_0} [f(X(k)) \mid \cF_i] \Eq (P^{k-i} f)(X(i)).
$$

To apply Lemma~\ref{thm.mart-b}, we need to bound the conditional variances $\var (Z_i \mid \cF_{i-1})$,
for $1 \le i \le k$.
Conditional on the event $X(i-1) = x_{i-1}$, $Z_i$ takes the value $(P^{k-i}f)(x)$ with probability $P(x_{i-1},x)$.
Since $\var Z \le \E\{(Z-c)^2\}$ for any $c\in\R$, it follows that
\begin{eqnarray}
  \lefteqn{\var (Z_i \mid X(i-1)= x_{i-1}) } \nonumber \\
    &&\Le  \sum_{x\in N(x_{i-1})} P(x_{i-1},x) \Bigl( (P^{k-i}f)(x) - c_{i-1} \Bigr)^2, \label{eq.var.1}
\end{eqnarray}
with $c_{i-1} := (P^{k-i}f)(x_{i-1})$.  Using Assumption~(\ref{cond-gen-4}), this yields
\eqa
    \lefteqn{\var (Z_i \mid X(i-1)= x_{i-1}) }\non\\
    &&\Le
           \sum_{x\in N(x_{i-1})} P(x_{i-1},x) \Big ( (P^{k-i}f)(x) - (P^{k-i}f)(x_{i-1}) \Big )^2 \nonumber \\
    &&\Le  \alpha_{k-i}, \label{eq.var}
\ena
uniformly in $x_{i-1} \in S$. It thus follows that
$$
    \sum_{i=1}^k \var (Z_i \mid \cF_{i-1}) \Le  \sum_{j=0}^{k-1} \alpha_j \Eq  a_k,
$$
so we set $\delta = a_k$.

We also need a uniform upper bound on $|Z_i - Z_{i-1}|$.
We note that
$$
Z_{i-1} = \E\bigl\{\E(f(X(k)) \giv \cF_i) \giv \cF_{i-1} \bigr\}
         =\! \sum_{z \in N(X(i-1))} P(X(i-1),z) (P^{k-i}f)(z).
$$
Note that, from Assumption~\Ref{cond-gen-5}, if $y,z \in N(x)$  for some $x \in S$, then
\eq\label{ADB-Pf-diff}
    \bigl|(P^i f)(y) - (P^i f)(z)\bigr| \Le 2 \beta.
\en
It then follows from~\Ref{ADB-Pf-diff} that, on the event $\{X(i-1)=x_{i-1}\}$,
\begin{eqnarray}
  \bigl| Z_i - Z_{i-1} \bigr|
            & = & \bigl| (P^{k-i}f)(X(i)) - \sum_{z \in N(x_{i-1})} P(x_{i-1},z) (P^{k-i}f)(z) \bigr| \nonumber \\
     & \le & \sum_{z \in N(x_{i-1})} P (x_{i-1},z) \bigl| (P^{k-i}f)(X(i)) - (P^{k-i}f)(z) \bigr| \nonumber \\
     & \le & 2 \beta, \label{eq.abs}
\end{eqnarray}
uniformly in $x_{i-1} \in S$, since, in the last sum, both~$X(i)$ and~$z$ belong to~$N(x_{i-1})$.
Accordingly, we take $\gamma = 2 \beta$.

Theorem~\ref{thm.concb-general} now follows from inequality~\eqref{ineq-cmcd-1} in Lemma~\ref{thm.mart-b},
in the case where $\wtS=S$.

In general, for each $i$, ~\eqref{eq.var} and~\eqref{eq.abs} hold if $x_{i-1} \in \wtS$, and so all the above bounds hold
on the event $A_k=\{X(i) \in \wtS \mbox{ for } i=0, \ldots, k-1\}$. Thus
$A_k \subseteq A(\delta,\gamma)$, as defined in Lemma~\ref{thm.mart-b},
and the full statement of Theorem~\ref{thm.concb-general} follows from inequality~\eqref{ineq-cmcd-2} in Lemma~\ref{thm.mart-b}.
\end{proofof}

\subsection{Contracting chains}

We next show how to use Theorem~\ref{thm.concb-general} to recover a version of Ollivier's results on chains with positive coarse Ricci curvature.

Let $d(\cdot, \cdot)$ be a metric on the state space $S$ of a discrete-time Markov chain $X = (X(i))_{i\ge0}$.
A Markovian coupling $(X\ui,X\ut)$ of two copies of the chain is {\em contracting} with respect to the metric if, for some positive
constant~$\rho$ and for all $x,y \in S$,
\eq\label{ADB-contract-disc-def}
    \E [d(X^{(1)}(1),X^{(2)}(1)) | (X^{(1)}(0), X^{(2)}(0)) = (x,y)] \le (1-\rho) d(x,y).
\en
If condition~(\ref{ADB-contract-disc-def}) holds for all $x,y$ in some subset $\wtS$ of $S$, then we say that the coupling is
contracting on $\wtS$.

The existence of a coupling satisfying~(\ref{ADB-contract-disc-def}) for all pairs of states is equivalent to the inequality
\eq\label{ollivier}
\sup_{x,y \in S} \frac{ W_d({\mathcal L}_x(X(1)), {\mathcal L}_y(X(1))) }{d(x,y)} \le 1-\rho,
\en
where $W_d$ denotes the Wasserstein distance between two measures with respect to the metric~$d$ on a space~$S$: $W_d(\mu,\nu)$ is the infimum of
$\E d(X,Y)$ over all pairs $(X,Y)$ of $S$-valued random variables, with ${\mathcal L}(X)  = \mu$ and
${\mathcal L}(Y) =\nu$.  Ollivier~\cite{Ollivier} defines a Markov chain to have {\em coarse Ricci curvature} at
least $\rho$ if~(\ref{ollivier}) holds: we prefer to say that the Markov chain is {\em contracting in Wasserstein distance}.

In the case where $d$ is a graph distance -- i.e., $d(x,y)$ is the length of a shortest path in a graph between vertices $x$ and $y$ --
inequality~(\ref{ollivier}) is equivalent to
\eq\label{GLP}
\sup_{x \sim y} W_d({\mathcal L}_x(X(1)), {\mathcal L}_y(X(1))) \le 1-\rho,
\en
where $\sim$ denotes adjacency in the graph.  Gheissari, Lubetzky and Peres~\cite{GLP} call a chain satisfying~(\ref{GLP}) $(1-\rho)$-contracting.
We prefer to use the term {\em contracting in Wasserstein distance} to avoid confusion with the concept of contraction introduced by Marton~\cite{Marton96},
which is contraction in total variation distance.

For a Markov chain that is contracting in Wasserstein distance with respect to a metric $d$,
we now prove concentration of measure for any real-valued function $f$ on the state space that is Lipschitz with respect to~$d$.
Part~(a) of the theorem below applies when the Markov chain is contracting on the entire state space;
part~(b) is for when the contraction is only on some ``good set''.

For an event $A$, we let $\overline A$ denote its complement.

\begin{theorem}\label{ADB-contract-disc-lem}
Let $X$ be a discrete-time chain on discrete state space~$S$ with transition matrix $P$.  Suppose that
$d(\cdot,\cdot)$ is a metric on $S$, and let $f\colon S \to \R$ be a function such that, for some constant~$L$,
$|f(x) - f(y)| \le L d(x,y)$ for all $x,y\in S$.  Suppose also that $D$ is a positive constant such that $d(x,y) \le D$ whenever $P(x,y) > 0$.

\smallskip

\noindent
(a) If~$X$ is contracting in Wasserstein distance, with constant~$\r$, and $D_2$ is a constant such that, for all $x \in S$,
\begin{equation} \label{Dtwo}
\sum_{y \in N(x)} P(x,y) d(x,y)^2 \le D_2
\end{equation}
then, for all $x \in S$, $m \ge 0$, and $k \in \N$,
\begin{eqnarray*}
  \Pr_{x} \Big ( \left| f(X(k))-\E_{x} [f(X(k))] \right|\ge m \Big) \phantom{ebweuybrubare} \\
    \Le 2\exp\left( - \frac {m^2}{2L^2D_2/(2\rho - \r^2) + 4LD m/3}\right).
\end{eqnarray*}

\noindent
(b) More generally, suppose that $X$ is contracting in Wasserstein distance on a subset ${\widehat S}$ of $S$, with constant $\r$, and let
$\wtS$ be a further subset of $S$ such that $\wtS^+ := \wtS \cup \bigcup_{x\in \wtS} N(x) \subseteq {\widehat S}$.  Suppose that (\ref{Dtwo}) holds for all $x \in \wtS$.
For $k$ a positive integer, let $A_k = \{X(j) \in \wtS \mbox{ for } 0 \le j \le k-1\}$, and define
$$
e_k := \sup_{y \in \wtS^+} \Pr_y( X(i) \notin {\widehat S} \mbox{ for some } i < k). 
$$
Then, for all $x \in \wtS$ and $m \ge 0$, 
\begin{eqnarray*}
\Pr_{x} \Big ( \{\left| f(X(k))-\E_{x} [f(X(k))] \right|\ge m \} \cap A_k \Big) \phantom{rqwygfyVYCw} \\
\le 2\exp\left( - \frac {m^2}{4L^2(D_2/\r + 12k^3 D^2 e_k^2)  + 4LDm(1+6ke_k)/3}\right).
\end{eqnarray*}
\end{theorem}

Note that we may always take $D_2 = D^2$, but sometimes it is possible to take $D_2$ significantly smaller.  In part~(b), we would expect to be able to choose the various
sets so that $e_k$ is very small.  In order to apply part~(b) effectively, one would need to know that $\Pr(\overline{A_k})$ is small, and this will not be true if the
starting state is ``close to the boundary'' of $\wtS$: a natural approach is to have {\em three} nested sets of states $S^* \subset \wtS \subset \widehat S$, with the
starting state restricted to $S^*$, and with the probability of escaping from one set to the next over the time interval of interest being small; then we obtain concentration
of measure over that time interval, uniformly over starting states in~$S^*$.

\begin{proof}
For part~(a), we apply Theorem~\ref{thm.concb-general} with $\wtS = S$.  For states $x$ and $y$ with $y \in N(x)$, let $(X^{(1)}(i))$ and $(X^{(2)}(i))$
be copies of the chain with $X^{(1)}(0) = x$ and $X^{(2)}(0) = y$, coupled so that $\E [ d(X^{(1)}(i),X^{(2)}(i)) ] \le d(x,y) (1-\rho)^i$ for each
$i \in \Z^+$.  Then we have
$$
|  (P^if)(x) - (P^if)(y)| = |\E f(X^{(1)}(i)) - \E f(X^{(2)}(i))| \le
$$
$$
\E |f(X^{(1)}(i)) - f(X^{(2)}(i))| \le L \E d(X^{(1)}(i),X^{(2)}(i)) \le L d(x,y) (1-\rho)^i,
$$
whenever $y \in N(x)$ and $i \in \Z^+$.  Thus we may take $\beta = LD$ in~(\ref{cond-gen-5}) and $\a_i = (1-\r)^{2i}L^2D_2$ in~(\ref{cond-gen-4}) for
each $i \in \Z^+$.
Since then $a_k \le L^2D_2/(2\r-\r^2)$ for all $k\ge1$, the inequality follows.

For part~(b), our plan is to apply Theorem~\ref{thm.concb-general} to the ``inner'' set $\wtS$, so we need bounds on
$|(P^if)(x) - (P^if)(y)|$ valid whenever $x \in \wtS$ and $y \in N(x) \subseteq \wtS^+$.  Accordingly, we fix such a pair $(x,y)$, and $k \in \N$.
We now consider two copies $(X^{(1)}(i))$ and $(X^{(2)}(i))$ of the chain, with $X^{(1)}(0)=x$ and $X^{(2)}(0)=y$, with a contractive coupling on $\widehat S$
with constant~$\rho$.
For $i \ge 1$, let $B_i$ be the event that both copies of the chain are in $\widehat S$ for all $j<i$, and note that $\Pr(\overline {B_i}) \le 2 e_i$.
We claim that, for each $i$,
$$
\E [ d(X^{(1)}(i),X^{(2)}(i)) I[B_i] ] \le d(x,y) (1-\rho)^i.
$$
This is true for $i=0$.  If the inequality is true for $i-1$, then
\begin{eqnarray*}
\lefteqn{\E [ d(X^{(1)}(i),X^{(2)}(i)) I[B_i] ]} \\
&=& \E [\E [ d(X^{(1)}(i),X^{(2)}(i)) I[B_i] \mid X^{(1)}(i-1), X^{(2)}(i-1) ]] \\
&\le& \E [ (1-\rho) d (X^{(1)}(i-1),X^{(2)}(i-1)) I[B_{i-1}] ] \\
&\le& (1-\rho) d(x,y) (1-\rho)^{i-1},
\end{eqnarray*}
as claimed.
As each step of either chain increases the distance between them by at most $D$, we also have the bound
$$
\E [ d(X^{(1)}(i),X^{(2)}(i)) I[\overline{B_i}] ] \le (2i+1) D \Pr (\overline{B_i}) \le 6i D e_i,
$$
for $i \ge 1$ and also for $i=0$, and therefore
$$
\E [ d(X^{(1)}(i),X^{(2)}(i)) ] \le (1-\rho)^i d(x,y) + 6iD e_i.
$$
Hence we have
$$
| (P^if)(x) - (P^if)(y)| \le L \big( (1-\rho)^i d(x,y) + 6i D e_i \big),
$$
whenever $x \in \wtS$ and $y \in N(x)$.  Additionally we have that
\begin{eqnarray*}
\lefteqn{\sum_y P(x,y) | (P^if)(x) - (P^if)(y)|^2 }\\
&\le& 2L^2 \Big( (1-\rho)^{2i} \sum_y P(x,y) d(x,y)^2 + 36i^2 D^2 e_i^2 \Big),
\end{eqnarray*}
for all $x \in \wtS$.
Thus we can apply Theorem~\ref{thm.concb-general} with $\b = LD ( 1 + 6k e_k)$, for $k \ge 1$, and $\a_i = 2L^2 ((1-\r)^{2i}D_2 + 36i^2 D^2 e_i^2)$
for each~$i$.  Since then $a_k \le 2L^2( D_2 /\r + 12 k^3 D^2 e_k^2)$, the inequality follows.
\end{proof}

Both parts of Theorem~\ref{ADB-contract-disc-lem} follow directly, with essentially the same proof as here, from Theorem~4.5 of Luczak~\cite{l08}.
Part~(a) of the result is also very similar to Theorem~33 of Ollivier~\cite{Ollivier}.  Ollivier's result is for the equilibrium distribution,
although he notes in Remark~39 that a similar result can be obtained for the finite-time distributions.  Ollivier's bounds are stated in terms of a
quantity called the coarse diffusion constant $\sigma(x)$, at a state~$x$, which is closely related to our $D_2$, and a quantity called the local
dimension $n_x$, that is of constant order in most applications with discrete state spaces.  Our proof of Theorem~\ref{thm.concb-general} could be
reworked to use the coarse diffusion constant directly (when bounding the conditional variances, we could instead use that
$\var(Z) = \frac12 \E (Z_1-Z_2)^2$, where $Z_1$ and $Z_2$ are independent copies of $Z$ -- see the proof of Lemma~4.6 in~\cite{l08}).  The conclusion
of our result translates to essentially the same as Ollivier's, with different constants.
The concentration result is of the ``Gaussian-then-exponential'' type.

\subsection{Approximately $f$-contracting chains}

We next illustrate how Theorem~\ref{thm.concb-general} can be applied in other settings, without even a metric on the state space.
One can obtain a result by analysing the direct effect a coupling has on the function $f$ of interest, if
the coupling is ``approximately $f$-contracting'', as we now describe.
As before, let $(X^{(1)})$ and $(X^{(2)})$ be two coupled copies of the Markov chain, and let $f: S \to \R$ be any function.
Suppose that $\sum_{y\in N(x)} P(x,y) |f(x)-f(y)|^2 \le F^2$ for any $x \in S$, and that, for all states $x,y \in S$,
\begin{eqnarray*}
\E [|f(X^{(1)}_1) - f(X^{(2)}_1) | | (X^{(1)}_0, X^{(2)}_0) = (x,y) ]
& \le & (1- \rho) |f(x) - f(y)| \\
&&\mbox{} +  \varepsilon (x,y),
\end{eqnarray*}
for some constant $\rho > 0$, and some ``error function'' $\varepsilon$.
(An example where there is a need for such an error function is in Lemma~3.1 of~\cite{BL2013}.)

An induction argument then gives that, for all $x,y \in S$ and every $k \in \N$,
$$
\E_{x,y} [|f(X^{(1)}_k) - f(X^{(2)}_k) ] \le (1-\rho )^k |f(x) - f(y)| + \eta_k (x,y),
$$
where
$$
\eta_k (x,y) = \sum_{i=0}^k (1-\rho)^{k-i} \E_{x,y} [\varepsilon (X^{(1)}_i, X^{(2)}_i) ].
$$
A convenient assumption, which is satisfied in the example from~\cite{BL2013}, is that
$\E_{x,y} [\varepsilon (X^{(1)}_i, X^{(2)}_i) ] \le  \eps_0 (1-\rho)^i$, for all $i$ and all $x,y \in S$ with $y \in N(x)$, so that $\eta_k(x,y) \le \eps_0 (k+1) (1-\rho)^k$
for each $k$ and each $x$ and $y$ with $y \in N(x)$.
It follows in this case that, for $x \in S$ and every $i \in \N$,
\begin{eqnarray*}
\lefteqn{\sum_{y \in N(x)} P(x,y) |(P^i f)(x) - (P^i f)(y)|^2} \\
 &\le& \sum_{y\in N(x)} 2 P(x,y) \big[(1-\rho)^{2i} |f(x)-f(y)|^2 + \eta_i (x,y)^2\big] \\
&\le& 2 F^2 (1-\rho)^{2i} + 2 \eps_0^2 (i+1)^2 (1-\rho)^{2i}.
\end{eqnarray*}
So we may take $\alpha_i = 2 (F^2 + \eps_0^2 (i+1)^2) (1-\rho)^{2i}$ in Theorem~\ref{thm.concb-general}, and hence $a_k = a = 2F^2/\rho + 4 \eps_0^2/\rho^3$
for all~$k$.
Also we may take $\beta = G + \eps_0$, where $G$ is a uniform bound on $|f(x) - f(y)|$ for all $x \in S$ and $y \in N(x)$.  Applying Theorem~\ref{thm.concb-general}
with these constants then gives a concentration inequality valid for all $x \in S$ and all $m \ge 0$:
$$
\pr_{x} \Big( |f(X_k) - (P^kf)(x)| \ge m  \Big) \le 2 e^{-m^2/(2 a + 4 \beta m/3)}.
$$

\subsection{A toy example}

Many of the chains we might be interested in have stationary distributions, and under suitable conditions our results on long-term concentration of measure imply
concentration of measure in equilibrium.  This is explored in Corollary~4.2 of Luczak~\cite{l08}, giving circumstances
where the chain is guaranteed to have a stationary distribution, and where concentration results carry over to equilibrium.  The main focus of the paper of
Ollivier~\cite{Ollivier} is also concentration of measure in equilibrium.  In the example in Section~\ref{supermarket} of this paper, we use facts from elsewhere about
the equilibrium distribution, as well as our long-term concentration results, to prove concentration of measure of a suitable function in equilibrium.

We finish this section with a very simple class of examples, illustrating very different circumstances when our results can be applied.  These examples have no
stationary distributions, and our results can be applied to show concentration of measure within a window whose width may be constant, or may increase with time.

Consider the discrete-time chain $X(k)$ with state space $\Z_+$, $X(0)=0$, and transition probabilities $p(i,i)=p(i,i+1)=1/2$.  This is thus a pure-birth chain,
stepping up with probability 1/2 at each time.  We also consider a function $f: \Z_+ \to \R$, and we are interested in the long-term behaviour of $f(X(k))$.
Of course, this is easy to analyse directly since $X(k)$ has a Binomial distribution with parameters $(k,1/2)$.  If, for example, $f(x) = x^r$ for some constant~$r \in (0,1]$,
then $f(X(k))$ is concentrated within a window of width $c k^{r-1/2}$ around $(k/2)^r$.

We start by explaining why the hypotheses of Theorem~\ref{ADB-contract-disc-lem} are too restrictive to encompass these examples.
Consider a coupling of two copies of the chain, so that at each step either both copies move up, or neither moves up.  (Choosing a different coupling would not make
any difference.)  Suppose that this coupling is contracting, with constant $\rho >0$, with respect to some metric $d$ on $\Z_+$.  Then we have
$$
\frac12 \left( d(i,j) + d(i+1,j+1) \right) \le (1-\rho) d(i,j),
$$
for each pair $(i,j)$, which amounts to $d(i+1,i) \le (1-2\rho) d(i,i-1)$ for each $i \ge 1$.  If the function $f$ is Lipschitz with respect to $d$, with
constant~$L$, then $|f(i+1) - f(i)| \le L (1-2\rho)^i d(1,0)$.  This condition is only satisfied if
$(f(i))$ converges to a limit $f_\infty$, and moreover $|f(i) - f_\infty| \le C (1-2\rho)_i$ for some constant~$C$.  In particular, none of the functions $f(x)=x^r$
satisfy the hypotheses, even though a time-independent concentration result does hold when $r \le 1/2$.

We now show how to apply our more general result, Theorem~\ref{thm.concb-general}, to the class of functions $f(x) = x^r$, with $0<r \le 1$.  We note that
$f(x+1) - f(x)$ is non-increasing in~$x$, and that
$\Pr(X(i) \le i/3) \le e^{-i/36}$ from the Chernoff bound.  Then we have, for any $x$, and $i$ sufficiently large,
\begin{eqnarray*}
|P^if(x+1) - P^if(x)| &\le& P^if(1) - P^if(0) \\
&\le& 1 \times \Pr(X(i) \le i/3) + \left[ (i/3 +1)^r - (i/3)^r \right] \\
&\le& e^{-i/36} + r (i/3)^{r-1} \le 2 (i/3)^{r-1}.
\end{eqnarray*}
Hence we may take $\alpha_i = 2 (i/3)^{2r-2}$ for large enough~$i$, and then $a_k = \sum_{i=0}^{k-1} \alpha_i$ is at most a constant $C(r)$ for $r \in (0,1/2)$, and at most
$C(r) k^{2r-1}$ for $r > 1/2$.  We may also take $\beta=1$.
For $r < 1/2$, applying Theorem~\ref{thm.concb-general} with $\wS$ equal to the entire state space $\Z_+$, gives a uniform bound on the concentration:
$$
\Pr(|f(X(k) - \E_0 f(X(k))| \ge m) \le 2 e^{-m^2/(2C(r) + 4m/3)},
$$
for all~$k$, showing that $f(X(k))$ remains concentrated within a window of constant width around its mean for all~$k$.  Of course, this is still far from a sharp result.
For $r > 1/2$, we obtain that
$$
\Pr(|f(X(k) - \E_0 f(X(k))| \ge m) \le 2 e^{-m^2/(2C(r)k^{2r-1} + 4m/3)},
$$
so that $f$ is concentrated within $k^{r-1/2}$ of its expectation, which in this case is the correct order of magnitude.

\section{Concentration inequality: continuous time}
\label{ADB-cts-time}

We now state and prove a 
continuous-time version of Theorem~\ref{thm.concb-general}.
For definitions concerning continuous-time Markov chains, see Anderson~\cite{Anderson}, in particular pages 13 and 81
(we use the term ``non-explosive'' in place of~``regular'').

Let $\wX=(\wX(t))_{t \in \R^+}$ be a stable, conservative, non-explosive continuous-time Markov chain with a discrete state space~$S$
and $Q$-matrix $(\wQ(x,y): x,y\in S)$.  Let $\wP^t = e^{\wQ t}$ denote the transition probabilities of $\wX$.
Much as before, for a function $f: S \to \R$, we write $(\wP^t f)(x)$ to denote $\E_x f(\wX(t))$, whenever it exists.

For $x \in S$, we set
\eqa
   N(x) &:=& 
   \{y \in S: \wQ(x,y) > 0\}.\non
\ena

%
%

\begin{theorem}
\label{thm.concb-continuous}
Let $(\wQ(x,y): x,y\in S)$ be the $Q$-matrix of a stable, conservative, non-explosive continuous-time Markov chain
$(\wX(t))_{t \ge 0}$ with discrete state space~$S$. 
Writing $q_x = -\wQ(x,x)$, let
$\wS$ be a subset of $S$, for which $q := \sup_{x\in \wS} \{q_x \} < \infty$. 
Let $f\colon S \to \R$ be a function such that $(\wP^t f)(x) := \ex_x f(\wX(t))$
exists for all $t \ge 0$ and $x \in S$, and suppose that $\wbeta$ is a constant such that
\begin{equation} \label{eq.coupling-cts}
   \big|(\grb{\wP^s} f)(x) -  (\grb{\wP^s} f)(y)\big| \Le \wbeta,
\end{equation}
for all $s \ge 0$, all $x \in \wS$ and all $y \in N(x)$.
Assume also that the continuous function $\weta: \R^+ \to \R^+$ satisfies
\begin{equation}\label{ADB-cts-key-bnd}
    \sum_{y\in S} \wQ(x,y) \big( (\wP^s f)(x) -  (\wP^s f)(y) \big)^2 \Le \weta(s),
\end{equation}
for all $x \in \wS$ and all $s \ge 0$. Define $\wa_t := \int_{s=0}^t \weta(s) \, ds$.
Finally, let
$A_t := \{\wX(s) \in \wS \mbox{ for all } 0 \le s < t\}$.
Then, for all $x_0 \in \wS$, $t\ge0$ and $m \ge 0$,
\begin{equation*}
  \Pr_{x_0} \Bigl( \left \{ \bigl|f(\wX(t))-(\wP^t f)(x_0) \bigr| > m \right \} \cap A_t \Bigr)
     \Le 2e^{-m^2/(2\wa_t + 2\wbeta m/3)}.
\end{equation*}
\end{theorem}

Exactly as in the discrete case, a bound on the variance of $f(\wX(t))$ follows in the case where $\wS= S$.

In order to prove the theorem, we first need to show that, for any fixed $x \in \wS$, the
function~$(\wP^s f)(x)$ has zero quadratic variation on any finite $s$-interval.
This follows from the following lemma.

\begin{lemma}\label{ADB-diffable}
 Under the above assumptions, for each $x \in \wS$, $(\wP^s f)(x)$ is continuously differentiable with respect to~$s$.
\end{lemma}

\begin{proof}
We can suppose that $f(x) \ge 0$ for all~$x\in S$;  if not, it suffices to consider
the positive and negative parts $f^+$ and~$f^-$ of~$f$ separately.  This enables the exchange
of sums and integrals in the argument that follows.

First, by considering what happens up to time $s$, we have
\[
     (\wP^t f)(x) \ \ge\ e^{-q_xs}(\wP^{t-s} f)(x),\qquad 0 \le s \le t,\ x \in \wS.
\]
Thus, from~\Ref{eq.coupling-cts}, for $x \in \wS$ and~$y \in N(x)$, it follows that
\eq\label{Pf-bound}
    (\wP^v f)(y) \Le \wbeta + e^{q_x(t-v)} (\wP^t f)(x), \qquad 0 \le v \le t  .
\en
Now, since 
$$
   \wP^s(y,z) = \pr_y[\wX(s) = z],
$$
the Kolmogorov backward equations imply that,
for any $x \in \wS$ and $s > 0$, we have
\begin{eqnarray}
 (\wP^s f)(x) &=& \sum_{y\in S} f(y)
         \Bigl\{e^{-q_xs}\delta_{xy} + \int_0^s e^{-q_xu} \sum_{z \in S}\wQ(x,z) \wP^{s-u}(z,y)\,du \Bigr\} \nonumber\\
  &=& f(x) e^{-q_xs} + \int_0^s  e^{-q_x(s-v)}\sum_{z \in S}\wQ(x,z) (\wP^v f)(z)\,dv . \label{2}
\end{eqnarray}
In view of~(\ref{Pf-bound}), and because $\sum_{z \in S}\wQ(x,z) = q_x < \infty$, the integrand on the
right hand side of~(\ref{2}) is uniformly bounded on $[0,t]$ for any~$t< \infty$, implying that the
indefinite integral is continuous in~$s$.  From this, it follows immediately that $(\wP^s f)(x)$ is
continuous in~$s$ also.  But then, for $x \in \wS$,
\[
    \sum_{z \in S} \wQ(x,z)(\wP^v f)(z) \Eq q_x(\wP^v f)(x) + \sum_{z \in S} \wQ(x,z)\{(\wP^v f)(z) - (\wP^v f)(x)\}
\]
is a {\it uniformly\/} convergent sum, in view of~\Ref{eq.coupling-cts}, and so the integrand in~(\ref{2}) is
continuous; thus the indefinite integral is continuously differentiable with respect to~$s$, and
hence $(\wP^s f)(x)$ is also.
\end{proof}

\begin{proofof}{Theorem~\ref{thm.concb-continuous}}
Fix $\wX(0)=x_0 \in \wS$ and, for $0 \le s \le t$, define
\[
     Z_s \Def \ex\{f(\wX(t)) \giv \FF_s\} - (\wP^t f)(x_0) \Eq (\wP^{t-s} f)(\wX(s)) - (\wP^t f)(x_0);
\]
note that $Z_t = f(\wX(t)) - (\wP^t f)(x_0)$ and that $Z_0 = 0$.  Then $(Z_s)_{0\le s\le t}$ is a martingale,
and so is $(\wZ_s)_{0\le s\le t}$, where $\wZ_s := Z_{s\wedge\t_0}$, and
\[
    \t_0 \Def \inf\{s \ge 0\colon\,\wX(s) \notin \wS\}.
\]
We now use a supermartingale derived from~$\wZ$ to prove a concentration bound.

In view of Lemma~\ref{ADB-diffable}, the continuous part of~$Z$ has no quadratic variation
until~$\t_0$, and so
the predictable quadratic variation of~$\wZ$ is given by
\[
   \langle \wZ \rangle_t \Eq \int_0^{t\wedge\t_0} \sum_y q(\wX(s),y) \{(\wP^{t-s}f)(y) - (\wP^{t-s}f)(\wX(s))\}^2\,ds.
\]
Hence, by~\Ref{ADB-cts-key-bnd},
\eq\label{1}
    \langle \wZ \rangle_t  \Le \int_0^t \weta(t-s)\,ds \ =\ \wa_t \ <\ \infty.
\en

Let the jump times of~$\wX$ be denoted by $0 < \s_1 < \s_2 < \cdots$, and write
\[
    U_s^h \Def \sum_{i\colon \s_i \le (s\wedge \t_0)} (e^{h\D Z_i} - 1 - h\D Z_i)
     \Eq h^2 \sum_{i\colon \s_i \le (s\wedge \t_0)} (\D Z_i)^2 g(h\D Z_i),
\]
where~$g(x)= (e^x-1-x)/x^2$, as in the proof of Lemma~\ref{thm.mart-b}, and, for $i$ such that $\s_i \le \t_0$,
\[
     \D Z_i \Def Z_{\s_i} - Z_{\s_i-}
           \Eq (\wP^{t-\s_i}f)(\wX(\s_i)) - (\wP^{t-\s_i}f)(\wX(\s_i-)),
\]
using the continuity of $(\wP^s f)(x)$ in~$s\ge0$ for each $x\in \wS$.

Let $V^h$ denote the compensator of~$U^h$.  We first note that~$V^h_s$ is finite, at least for $s \le \t_0$.
This is because, for $0 \le v < s \le \t_0$, we have
\[
    0 \Le U^h_s - U^h_v \Le h^2 g(h\wbeta) \sum_{i\colon v < \s_i \le s}(\D Z_i)^2 \quad {\rm a.s.},
\]
by~\Ref{eq.coupling-cts}, as $g$ is increasing on $[0,\infty)$.  Hence, noting that $A_t = \{\t_0 \ge t\}$, we see that
\eq\label{4}
I[A_t] e^{V_t^h} \le I[A_t] \exp( h^2 g(h \wbeta) \wa_t),
\en
in view of~(\ref{1}).

Now $\wZ$ is a square integrable martingale, because of~(\ref{1}), and hence, from the proof
of Lemma~2.2 in van de Geer~\cite{vdG95},
$\exp\{h\wZ_s - V^h_{s\wedge\t_0}\}$ is a non-negative supermartingale with initial value~$1$,
since the continuous part of~$\wZ$ has no quadratic variation. Thus
\[
    1 \ \ge\ \ex(I[A_t] \exp\{h\wZ_t - V^h_{t\wedge\t_0}\}) \Eq \ex(I[A_t] \exp\{hZ_t - V^h_{t}\}).
\]
On the other hand, using 
(\ref{4}),
\[
    I[A_t] \exp\{hZ_t - V^h_{t}\} \ \ge\ I[A_t] e^{hZ_t} \exp\{-h^2 g(h\wbeta) \wa_t\}.
\]
Hence
\[
    e^{hm}\pr_x[\{Z_t \ge m\} \cap A_t] \Le \ex_x\{I[A_t] e^{hZ_t}\} \Le \exp\{h^2 g(h\wbeta) \wa_t\},
\]
or
\[
    \pr_x[\{f(\wX(t)) - (\wP^t f)(x) \ge m\} \cap A_t] \Le \exp\{h^2 g(h\wbeta) \wa_t - hm\}.
\]
We again optimise in $h$, as in the proof of Theorem 2.7 in~McDiarmid~\cite{cmcd98}, and then repeat the argument for a bound on
$\pr_x[\{f(\wX(t)) - (\wP^t f)(x) \le - m\} \cap A_t]$.
\end{proofof}

Let $(\wX(t))_{t\ge0}$ be a stable, conservative, non-explosive continuous-time chain with state space $S$,
and let $d(\cdot,\cdot)$ be a metric on $S$.  A Markovian coupling of two copies of $(\wX(t))_{t\ge0}$ is itself a contiuous-time Markov chain, with a generator
that we denote $\cA$.  The coupling is said to be
{\em contracting with respect to $d$, with constant $\rho > 0$}, if, for all $x,y \in S$,
\eq\label{ADB-contract-cts-def}
   \cA d(x,y)  \Le -\rho d(x,y).
\en
If the above holds for all $x$ and $y$ in some $\wS \subseteq S$, then we say that the coupling is contracting on $\wS$.
We say that $(\wX(t))_{t\ge0}$ is {\em contracting in Wasserstein distance} if there is a coupling satisfying~(\ref{ADB-contract-cts-def}) for all
$x,y \in S$.  This definition corresponds to that of positive coarse Ricci curvature for continuous-time chains given by Veysseire~\cite{Veysseire},
in the setting of jump chains.

The next result establishes concentration of measure for continuous-time chains that are contracting in Wasserstein distance.
We state our result only for the case when the Markov chain is contracting on the
entire state space, but there is not necessarily a global upper bound on the total transition rate out of a state.  We could also provide a
version for use when the contraction property only holds on a ``good set'', but it seems hard to cover all the possible cases where such a result
might be useful: an issue is that we need some mild control on the growth of $f$ in the unlikely event that the chain leaves the good set (in the discrete
case, we used that the chain makes a bounded number of steps of bounded distance) and the form of the bounds will depend on the manner of that control.

\begin{theorem}\label{ADB-contract-cts-lem}
Let $\wX$ be a stable, conservative, non-explosive
continuous-time Markov chain on a discrete state space $S$, with $Q$-matrix $\wQ := (\wQ(x,y): x,y\in S)$.  Suppose that
$d(\cdot,\cdot)$ is a metric on $S$, and let $f\colon S \to \R$ be a function such that, for some constant~$L$,
$|f(x) - f(y)| \le L d(x,y)$ for all $x,y\in S$.

Let $\wS$ be a subset of $S$, and let $q$ and $D$ be constants such that $-\wQ(x,x) \le q$ for all $x \in \wS$ and
$d(x,y) \le D$ whenever $x \in \wS$ and $y \in N(x)$.
For $t > 0$, let $A_t = \{\wX(s) \in \wS \mbox{ for } 0 \le s <t\}$.

Suppose that $\wX$ is contracting in Wasserstein distance, as in~\Ref{ADB-contract-cts-def}, with constant~$\r$.
Then, for all $x \in \wS$, $t > 0$ and $m \ge 0$,
\begin{eqnarray*}
\lefteqn{\Pr_{x} \Big ( \Big\{\left| f(\wX(t))-\E_{x} [f(\wX(t))] \right|\ge m \Big\} \cap A_t \Big)} \\
    &\Le& 2\exp\left( - \frac {m^2}{qL^2D^2/\rho  + 2LD m/3}\right).
\end{eqnarray*}
\end{theorem}

\begin{proof}
It follows from~\Ref{ADB-contract-cts-def} that, under a contracting coupling of two copies $\wX^{(1)}$ and $\wX^{(2)}$, the
process $\bigl\{e^{\r t}d(\wX^{(1)}(t), \wX^{(2)}(t))\bigr\}_{t \ge 0}$ is a non-negative local supermartingale.
Thus, if $(\wX^{(1)}(0), \wX^{(2)}(0)) = (x,y)$, then
\eq\label{ADB-contract-cts-basic}
\E d(\wX^{(1)}(t), \wX^{(2)}(t)) \le e^{-\rho t} d(x,y),\qquad t\ge0.
\en
We can now apply Theorem~\ref{thm.concb-continuous}, with
$$
  \wbeta   \Eq  L D, \quad
  \weta(s) \Eq e^{-2\rho s} q L^2 D^2,
$$
and so, for any $t > 0$,
$$
   \wa_t \Eq  q L^2 D^2 \int_{0}^t e^{-2\rho s} \, ds \Le \frac{q D^2 L^2}{2\rho}.
$$
The result now follows from Theorem~\ref{thm.concb-continuous}.
\end{proof}

Note that the upper bound in Theorem~\ref{ADB-contract-cts-lem} on the deviations of $f(\wX(t))$ from its expectation does not depend on~$t$.  As in the discrete case,
in many applications, the distribution of $\wX(t)$ will approach an equilibrium, and the bound above implies a bound on the concentration of $f(\wX(t))$ in equilibrium.
However, it might well be the case that $\Pr(A_t) \to 0$ as $t \to \infty$: eventually the chain leaves the good set, and once it does we cannot hope to say much about its
behaviour.


\section{Upper bounds on coalescence times}
\label{coalesce}

In this section, we prove an auxiliary result for continuous-time Markov chains, which we will use (primarily in Section~\ref{parasites}) to show that
a chain with a contracting coupling mixes rapidly once it enters a region $R$ of the state space where the equilibrium distribution is concentrated; this is
therefore a useful ingredient in a proof of cut-off, showing that the mixing time from any ``distant'' state is dominated by the ``travel time'' to reach $R$.

We study a function of a continuous-time Markov chain on the non-negative reals, with non-positive drift in all positive states,
and prove a lower bound on the hitting time of state~0.  For a contracting coupling $(X(t),Y(t))$ of two copies of a Markov chain with respect to the metric~$d$ on their
state space~$S$, we can apply our result below to the function $d(X(t),Y(t))$ of the Markov chain $(X(t),Y(t))$, in order to show that coalescence occurs quickly once
the distance between the two copies is reasonably small: we illustrate this method in Section~\ref{parasites}.

We deal only with the continuous-time case.  Proposition~17.19 of Levin, Peres and Wilmer~\cite{LPW} gives an analogous result for discrete-time chains, which can often
be used in a similar way to that described above; our proof of the proposition below follows theirs.

\begin{proposition} \label{hitting}
Let $X$ be a stable, conservative, non-explosive continuous-time Markov jump chain, with state space $S$ and $Q$-matrix $Q$.
Let $B$ and $\sigma^2$ be positive, and let $f\colon S \to \R_+$ be a function.  Set $S_0 := \{x\colon f(x) = 0\}$, and assume that:
\begin{itemize}
\item [(i)] the drift $\sum_y Q(x,y) \big( f(y) - f(x) \big)$ of $f$ is non-positive for all $x$ in $S \setminus S_0$;
\item [(ii)] $f(X)$ makes jumps of magnitude at most $B$;
\item[(iii)] $\sum_y Q(x,y) \big( f(y) - f(x) \big)^2 \ge \sigma^2$ for all $x\in S \setminus S_0$.
\end{itemize}
Define $T_* := \inf \{ t \colon f(X(t)) = 0 \}$, the hitting time of $S_0$.
Then, for any $t_0 \ge 2 B^2/\sigma^2$,
\begin{equation} \label{eq.jht}
\Pr (T_* \ge t_0) \le \frac{2\sqrt2 f(X(0))}{\sigma \sqrt{t_0}}.
\end{equation}
\end{proposition}

Notes:
\begin{enumerate}
\item [(a)] The nature of the underlying state space $S$ is not relevant, and we do not need to assume that the set $\{ f(x) : x \in S\}$ is discrete.

\item [(b)]
It is not a priori obvious that $S_0$ is non-empty or that $T_*$ is a.s.\ finite, but these follow from the result.

\item [(c)]
Suppose that $f(X_0) \ge B/2$.  In the case where $t_0 < 2 B^2/\sigma^2$, we then have
$\Pr (T_* \ge t_0) \le 1 \le \frac{2\sqrt2 f(X(0))}{\sigma \sqrt{t_0}}$, and so (\ref{eq.jht}) holds without any condition on $t_0$.
\end{enumerate}

The motivating example underlying the proposition is that of a simple random walk $X(t)$ on $\Z_+$
(with $f(x) = x$), making steps up and down each at rate $1$, until the walk hits~0, so that the sum in~(iii)
is equal to $2$ for each positive state.  In this case, the proposition says that
the walk hits~0 before time~$t_0$ with probability at least
$1 - \frac{2 X(0)}{\sqrt{t_0}}$, which is best possible up to a constant factor.  The proposition
then gives conditions, for more general processes, under which the same behaviour holds.

As mentioned already, we shall apply Proposition~\ref{hitting} to a Markovian coupling $(X,Y)$, where $X$ and $Y$ are two copies of a jump Markov chain with a state
space~$S$ equipped with a metric $d$, and $f\big((x,y)\big) = d(x,y)$.  The conclusion is equivalent to saying that the chains have coalesced by time $t_0$ with probability
at least $1 - 2\sqrt 2 d(X(0),Y(0))/\sqrt{t_0} \sigma$ (unless the two chains start within distance $B/2$ of each other, where $B$ is the maximum size $B$ of a jump in the
distance, and $t_0$ is less than $2 B^2/\sigma^2$).  If the coupling is contracting with respect to $d$, then condition~(i) is
satisfied.  A lower bound $\sigma^2$ on the expression in condition~(iii) can be obtained when, under the coupling, the distance between the two copies changes by at least
$\eta$ at rate at least~$r$, for suitable $\eta$ and $r$.

Our proof follows that of Proposition~17.19 in Levin, Peres and Wilmer~\cite{LPW}.

\begin{proof}
Let $D(t) = f(X(t))$, so that $T_* = \inf \{ t : D(t) = 0\}$.  For some $h \ge B \vee D(0)$ to be chosen later, let
$T_h = \inf \{ t : D(t) = 0 \mbox{ or } D(t) \ge h\}$.  We note that, for any $t_0 \ge 0$,
$$
\Pr (T_* \ge t_0) \le \Pr (T_h \ge t_0) + \Pr (D(t_0 \wedge T_h) \ge h).
$$
We now give bounds on the two terms on the right above.

By~(i), the process $(D(t \wedge T_h))$ is a supermartingale, and by (ii) it is bounded between $0$ and $h+B$.
Therefore, by the Optional Stopping Theorem, we have
$D(0) \ge \E D(t_0 \wedge T_h) \ge h \Pr(D(t_0 \wedge T_h) \ge h)$, and so $\Pr (D(t_0 \wedge T_h) \ge h) \le D(0)/h$.

For $t \ge 0$, we set $G(t) = D(t)^2 - 2h D(t) - \sigma^2 t$.  We claim that $(G(t \wedge T_h))$ is a submartingale.
For $s < t \wedge T_h$, we have
\begin{eqnarray*}
\lefteqn{\E[ G(t\wedge T_h) \mid X(s) ]}  \\
&=& G(s) + \E \int_{u=s}^{t\wedge T_h} \sum_y Q(X(u),y) \Big( f(y)^2 - f(X(u))^2 \\
&&\mbox{} - 2h \big( f(y) - f(X(u)) \big) \Big) - \sigma^2 \, du
\end{eqnarray*}
As
\begin{eqnarray*}
\lefteqn{f(y)^2 - f(X(u))^2 - 2h \big( f(y) - f(X(u)) \big)} \\
&=& \big( f(y) - f(X(u)) \big)^2 - 2 \big(h - f(X(u))\big) \big( f(y) - f(X(u)) \big),
\end{eqnarray*}
we have
\begin{eqnarray*}
\lefteqn{\sum_y Q(X(u),y) \big( f(y)^2 - f(X(u))^2 - 2h \big( f(y) - f(X(u)) \big) \big)} \\
&\ge& \sum_y Q(X(u),y) \big( f(y) - f(X(u)) \big)^2 \\
&& \mbox{} - 2 \big(h - f(X(u))\big) \sum_y Q(X(u),y) \big( f(y) - f(X(u)) \big) \\
&\ge& \sigma^2
\end{eqnarray*}
for all $u < T_h$, by (i) and~(iii), and so indeed $\E[ G(t\wedge T_h) \mid X(s) ] \ge G(s)$ for $s < t \wedge T_h$.

For $t \le T_h$, we have $2 h D(t) - D(t)^2 =(2 h - D(t)) D(t) \ge 0$, as $0 \le D(t) \le h + B \le 2h$ (since $h \ge B$) for $t \le T_h$.  Thus we have, for any
$t \ge 0$, $\E (2 h D(t \wedge T_h) - D(t \wedge T_h)^2) \ge 0$, and so
\begin{eqnarray*}
2 h D(0) &=& D(0)^2 - G(0) \ge - \E G (t \wedge T_h) \\
&=& \E \big (2 h D(t \wedge T_h) - D(t \wedge T_h)^2) \big) + \sigma^2 \E (t \wedge T_h) \\
&\ge& \sigma^2 \E(t \wedge T_h).
\end{eqnarray*}
Hence we obtain, for any $t\ge 0$, $\E (t \wedge T_h) \le 2hD(0)/ \sigma^2$.  Letting $t$ tend to infinity and applying the
Monotone Convergence Theorem, we obtain the same upper bound on $\E T_h$.  Therefore, for any $t_0 > 0$,
$$
\Pr (T_h \ge t_0) \le \frac{2h D(0)}{\sigma^2 t_0}.
$$

We conclude that
$$
\Pr (T_* \ge t_0) \le \frac{2h D(0)}{t_0 \sigma^2} +  \frac{D(0)}{h}.
$$
Optimising this bound by setting $h = \sigma \sqrt{t_0/2}$ now gives, provided $t_0 \ge 2 (B \vee D(0))^2/\sigma^2$ (so that $h \ge B \vee D(0)$),
$$
\Pr (T_* \ge t_0) \le \frac{2\sqrt 2 D(0)}{\sigma \sqrt{t_0}}.
$$
If $D(0) > \sigma \sqrt{t_0/2}$, then the result is trivial, so we obtain the bound above under the condition $t_0 \ge 2 B^2/\sigma^2$.
\end{proof}

We remark that the assumption of bounded jumps cannot be dropped.  Let $(X(t))$ be a chain on $\mathbb Q$ with $Q$-matrix $Q$ given by (a)~for $x<1$, $Q(x,x/2) = 1$ and
$Q(x,x+1/x) = x^2/2$, and (b)~for $x\ge 1$, $Q(x,x+1/2) = Q(x,x-1/2) = 1$.  Then $(X(t))$ is a non-explosive jump chain satisfying conditions~(i) and~(iii) with
$\sigma^2 = 1/2$.  From a state $x < 1$, the probability that all subsequent jumps are down is equal to $\prod_{k=0}^\infty 1/(1 + x^2/2^{2k+1}) > 0$.
Thus the chain makes a.s.\ finitely many visits to $[1,\infty)$ before entering $(0,1)$ and making only downward jumps thereafter, but $(X(t))$ can never reach~0.

Alternatively, consider the chain on $\mathbb Q$ with a $Q$-matrix such that $Q(x,x+1) =1$ for all $x$, $Q(x,x/2) = 2/x$ for $x\le 2$,
and $Q(x,x-1) = 1$ for $x \ge 2$.  This chain satisfies all of (i)-(iii), with $\sigma^2 = 1$, but is explosive: starting from a state~$x\le 2$, the probability that
the chain makes infinitely many downward jumps before the first upward jump is $\prod_{k=1}^\infty 2^k/(2^k + x) > 0$.  State~0 is not reached before the
explosion time.

\section{Bernoulli--Laplace diffusion model}\label{DS}

As our first example, we re-examine the Bernoulli--Laplace chain (Feller~\cite{Feller68}, Example XV.2(f)), for which cut-off was
first established in Diaconis and Shahshahani~\cite{ds87}.  In this model, there are two urns, the left urn initially containing~$n$ red balls, and the
right urn~$n$ black balls.  Then, at each time step, a ball is chosen at random in each urn, and the two balls are switched.

The state of the system at any time~$r \ge 0$ is captured by the number $X^{(n)}(r)$
of red balls in the left urn at time~$r$.  The chain~$X\un$ can be viewed as a discrete-time lazy random walk
with state space $\{0, \dots, n\} \subset \Z$, with state-dependent transition probabilities
\begin{eqnarray*}
  \Pr[X^{(n)}(r+1) = j+1 | X^{(n)}(r) = j]  & = & (1-j/n)^2,\\
  \Pr [X^{(n)}(r+1) = j-1 | X^{(n)}(r) = j] & = & (j/n)^2;\\
  \Pr [X^{(n)}(r+1) = j | X^{(n)}(r) = j]   & = & 1 - (1-j/n)^2 - (j/n)^2.
\end{eqnarray*}

Diaconis and Shahshahani examine the total variation distance between the distribution of $X^{(n)}(r)$ and its equilibrium distribution $\pi = \pi\un$,
a hypergeometric distribution with parameters $(2n, n, n)$, defined by
$$
   \pi^{(n)} (j) \Def \binom{n}{j} \binom{n}{n-j}\Big/ \binom{2n}{n}, \quad 0 \le j \le n.
$$
Analogously to earlier, we use $\law_j$, $\pr_j$ and~$\ex_j$ to refer to distributions conditional on $X\un(0)=j$, and we
also use $\law_{\pi^{(n)}}$, $\pr_{\pi^{(n)}}$ and~$\ex_{\pi^{(n)}}$ to refer to the equilibrium distribution.

Letting $r_n(\d) := \lfloor \quarter n \log n + \d n \rfloor$, Diaconis and Shahshahani~\cite{ds87}
show that there are universal constants $C_1,C_2 > 0$ such that
\eqa
    \dtv\bigl(\law_n(X\un(r_n(\d))),\pi\un \bigr) &\ge& 1\!-\! C_1e^{4\d}, \, -\quarter\log n \le \d < 0; \non\\
    \dtv\bigl(\law_n(X\un(r_n(\d))),\pi\un \bigr)  &\le& C_2e^{-2\d}, \quad \d \ge 0.\label{DS-2}
\ena
Their proofs, especially that of~(\ref{DS-2}), are based on algebraic techniques.  Although they only consider starting from state~$n$, which is easily seen to
maximise the mixing time, their proofs extend readily to cover other starting states.  The upper bound~(\ref{DS-2}) holds for any starting state.  If the chain
is started in a state $j$ in
$$
E_n(\eps) := \{ j : |j-\frac n2| \ge \eps n\},
$$
then a minor adjustment to their proof yields a bound of the form
\begin{equation} \label{DS-1b}
\dtv\bigl(\law_j(X\un(r_n(\d))),\pi\un \bigr) \ge 1\!-\! C_3e^{4\d}\eps^{-2}, \, -\quarter\log n \le \d < 0,
\end{equation}
for some universal constant $C_3$.

Thus, in the language introduced in Section~\ref{S:intro}, we have the following result.

\begin{theorem}\label{BL-thm}
For any $\eps > 0$, the Bernoulli--Laplace chain exhibits cut-off at $\quarter n\log n$ on $E_n (\eps)$ with window width~$n$.
\end{theorem}

We use the results of the previous sections to give an alternative, coupling proof of Theorem~\ref{BL-thm}, yielding the bounds in the result below.

\begin{theorem} \label{thm.BL2}
Let $X^{(n)}(r)$ be a copy of the Bernoulli-Laplace chain.
For $\delta \in \R$, set $r_n(\d) := \lfloor \quarter n \log n + \d n \rfloor$.

(a) For $-\quarter\log n \le \d < 0$, we have
\begin{equation*} 
\dtv\bigl(\law_j(X\un(r_n(\d))),\pi\un \bigr) \ge 1\!-\! 4\exp\Big( - \frac{\eps^2}{32} e^{-4\d}\Big),
\end{equation*}
for any $\eps > 0$, any $j \in E_n(\eps)$, and $n \ge 4$.

(b) For $0 \le \d \le \quarter \log n - \log \log n$, we have
\begin{equation} \label{DS-2a}
\dtv\bigl(\law_j (X\un(r_n(\d))),\pi\un \bigr) \le 21 e^{-2\d},
\end{equation}
for any $j \in \{ 0, \dots, n\}$, and $n$ sufficiently large.
\end{theorem}

Thus our upper bound in Theorem~\ref{thm.BL2}(b) matches that of Diaconis and Shahshahani in (\ref{DS-2}), except that our proof requires a mild upper bound on $\d$,
and our lower bound in part~(a) improves on~(\ref{DS-1b}).  The inequalities above are more than enough to imply Theorem~\ref{BL-thm}.

Extensions and generalisations of the result of Diaconis and Shahshahani have also been obtained.
For instance, Donnelly, Lloyd and Sudbury~\cite{DLS94} showed cut-off for the separation distance mixing time for this model,
and recently Eskenazis and Nestoridi~\cite{EN} showed cut-off for the version where $k>1$ balls are exchanged at each step.  All of these papers
make some use of algebraic techniques.

\medskip

We now give a brief overview of our proof of Theorem~\ref{thm.BL2}.  The first step is to use our discrete-time concentration of measure inequality,  Theorem~\ref{ADB-contract-disc-lem}(a), to show that, for any starting state $j=X^{(n)}(0)$ and any $r$, $X^{(n)}(r)$ is well-concentrated around its mean.
An easy estimate for the mean then shows that, with high probability, $X^{(n)}(r)$ is far from $n/2$ for $r \le r_n(0)$, and this is enough to give part~(a).

The proof of~(b) is more complicated.  The concentration of measure result shows that $X^{(n)}(r)$ is unlikely to leave a neighbourhood of $n/2$ for a long period of
time after $r_n(0)$; while it is in this neighbourhood, we can approximate the transitions of the chain by the transitions of a simpler chain whose long-term behaviour is easy to analyse, and show that the two chains therefore have approximately the same distributions over a suitably long time interval.

\medskip

We proceed by stating and proving a sequence of lemmas.
In what follows we drop the superscript~$(n)$, writing $X(r)$ instead of $X\un(r)$, to lighten the notation.


\begin{lemma} \label{lemma1}
Let $X(r) = X^{(n)}(r)$ be a copy of the Bernoulli-Laplace chain, with $n \ge 4$.
For all starting states $j \in \{ 0, \dots, n\}$, all $r \in \Z_+$, and all $c$ with $0 \le c \le 3 \sqrt n /4$, we have
$$
   \Pr_j (|X(r) - n x_j(r) | \ge c \sqrt n ) \Le 2 e^{-c^2/2},
$$
where
\begin{equation} \label{ADB-new-mean}
x_j(r) = \E_j X(r) / n = \Big ( \frac{j}{n} - \frac12 \Big ) \Big (1 - \frac{2}{n} \Big )^r + \frac12.
\end{equation}
\end{lemma}

\begin{proof}
%
Our plan is to use Theorem~\ref{ADB-contract-disc-lem}, and accordingly our first step is to describe a contractive coupling.

We fix $n \ge 4$, and $j_0 \in \{0, \ldots, n-1\}$, and let $(X^{1}(r))$ and $(X^{2}(r))$ be two copies of the chain starting
in $j_0$ and $j_0+1$ respectively.  We describe a coupling of the chains such that $|X^1(r) - X^2(r)|$ remains equal to~1 until dropping to~0.
When the two chains are in adjacent states $j$ and $j+1$ with $1 \le j \le n-2$, say with $X^{1}(r) = j$ and $X^{2}(r) = j+1$, then the next step of
the coupling is as follows.  The two chains
jump together up by 1 with probability $(1-(j+1)/n)^2$ and down by 1 with probability $(j/n)^2$.
Additionally, the lower chain $X^{1}(r)$ jumps up by 1 alone with probability $(1-j/n)^2 - (1-(j+1)/n)^2 = (2n-2j-1)/n^2$, and the higher chain $X^{2}(r)$
jumps down by 1 alone at rate $((j+1)/n)^2 - (j/n)^2 = (2j+1)/n^2$.  This leaves probability $\frac{1}{n^2}\big( (n-j)^2 + (j+1)^2 \big)$
that both chains stay in their current state.  Note that indeed $X^2(r+1) - X^1(r+1)$ is either~1 or~0, and that
$$
\Pr(X^2(r+1) = X^1(r+1) \giv X^1(r) = j, X^2(r) = j+1)
$$
$$
= \frac{2n-2j-1}{n^2} + \frac{2j+1}{n^2} = \frac{2}{n},
$$
for $1\le j \le n-2$.

The rules above do not define a coupling in the case where $j=0$ or $j=n-1$.  In the case $j=0$, for instance, $X^1(r)$ jumps from~0 to~1 with probability~1,
and $X^2(r)$ jumps to one of~0, 1, or 2 with probabilities $(1/n)^2$, $2/n - 2/n^2$, and $(1-1/n)^2$ respectively.  There is thus no {\em monotone} coupling
possible.  However, when $X^1(r) = 0$ and $X^2(r) = 1$, the next step of the coupling is forced since $X^1(r+1) = 1$ with probability~1, and it is still the case
that $|X^2(r+1) - X^1(r+1)|$ is either~1 or~0.  We have
$$
\Pr(X^2(r+1) = X^1(r+1) \giv X^1(r) = 0, X^2(r) = 1) = \frac{2}{n} - \frac{2}{n^2},
$$
and similarly for $j=n-1$.  Hence our coupling is contractive with constant $\rho = 2/n - 2/n^2$.

We take $f(x) =x$ in Theorem~\ref{ADB-contract-disc-lem}(a), with $S = \{0, \dots, n\}$, $d(x,y) = |x-y|$, $L = D = D_2 = 1$, and
$\r = 2/n - 2/n^2$, so that $2/(2\r - \r^2) \le n$ for all $n \ge 4$.
Then, by Theorem~\ref{ADB-contract-disc-lem}(a), for all $j \in \{ 0, \dots, n\}$, all $r \in \Z_+$, and all $m > 0$, we have
$$
   \Pr_j (|X(r) - \E_j X(r) | \ge m ) \Le 2 e^{-m^2/(n + 4 m/3)}.
$$
If we set $m = c \sqrt n$, for $0 \le c \le 3 \sqrt n /4$, we obtain that $n +4m/3 \le 2n$, and so
$$
   \Pr_j (|X(r) - \E_j X(r) | \ge c \sqrt{n} ) \Le 2 e^{-c^2/2}.
$$

To complete the proof, it remains to verify the formula for $x_j(r) := \E_j X(r)/n$.  Observe that
$$
   \adbng{\E_j X(r+1) \Eq \E_j X(r) + \E_j (1-X(r)/n)^2 - \E_j (X(r)/n)^2,}
$$
so that
$$
    x_j(r+1) \Eq 1/n + x_j(r) (1-2/n),
$$
and hence
$$
   x_j(r) \Eq \Big ( \frac{j}{n} - \frac12 \Big ) \Big (1 - \frac{2}{n} \Big )^r + \frac12,
$$
as claimed
\end{proof}


A matching tail bound for the equilibrium distribution $\pi^{(n)}$ follows from Lemma~\ref{lemma1}.  In fact, unsurprisingly, sharper tail bounds on the
hypergeometric distribution are known: results of Hoeffding~\cite{hoeffding} (see Section~6 and Theorem~1) imply that, for any $c \ge 0$,
\begin{equation} \label{Chv}
\Pr_{\pi^{(n)}}(|X - \frac n 2| \ge c \sqrt n) \le 2 e^{-2c^2}.
\end{equation}
An alternative proof was given by Chv\'atal~\cite{Chv}.

It is now not hard to obtain the claimed lower bound on total variation distance for $r < \quarter n \log n$.

\medskip

\begin{proof} [Proof of Theorem~\ref{thm.BL2}(a)]
For $r = r_n(\delta) = \lfloor \quarter n \log n + \delta n \rfloor$,
and $\delta <0$, we have seen that both $X(r)$ and the equilibrium distribution are well-concentrated around their respective means.  We will show that,
if $j = X(0)$ is in $E_n(\eps)$ for some fixed $\eps >0$, so that $| j - \frac n2 | \ge \eps n$, then the means are still far apart at time~$r$.


From~(\ref{ADB-new-mean}), we have that, uniformly in $ -\quarter\log n \le \d \le 0$,
\begin{eqnarray}
   \Big |x_j(r_n(\delta)) - \frac12  \Big | &\ge& \eps \Big (1 - \frac{2}{n} \Big )^{\quarter n\log n + \d n}
           \ \ge\ \eps e^{2|\d|}\Big (1 - \frac{2}{n} \Big )^{\quarter n\log n} \nonumber \\
           &\ge& \half \eps n^{-1/2}e^{2|\d|}, \label{ADB-mean-bounds-lower}
\end{eqnarray}
for all $n \ge 4$ (so that $n^{1/2}\Big (1 - \frac{2}{n} \Big )^{\quarter n\log n} \ge 1/2$).

For fixed $\eps >0$ and $\delta$ with $ -\quarter\log n \le \d \le 0$, we set
$$
A := \Bigl[\frac{n}{2} - \frac14 \eps e^{2|\delta|} n^{1/2}, \frac{n}{2} + \frac14 \eps e^{2|\delta|} n^{1/2}\Bigr].
$$
By~(\ref{Chv}), we have
$$
  1- \pi^{(n)}(A) = \Pr_{\pi^{(n)}} \Bigl(|X - \E_{\pi^{(n)}} (X)| > \frac14 \eps e^{2|\delta|} n^{1/2} \Bigr)
\le 2 \exp( -\frac18 \eps^2 e^{4|\delta|}).
$$

Similarly, using \Ref{ADB-mean-bounds-lower} and Lemma~\ref{lemma1}, we have that, for any $j \in E_n (\eps)$,
\begin{eqnarray*}
       \Pr_j (X(r_n(\delta)) \in A) &\le& \Pr_j \Bigl(|X(r_n(\delta)) - \E_j (X(r_n(\delta)))| > \frac14 \eps e^{2|\delta|} n^{1/2} \Bigr) \\
       & \le& 2 \exp( -\frac{1}{32} \eps^2 e^{4|\delta|} ),
\end{eqnarray*}
for all $n \ge 4$.  Hence we have
\begin{eqnarray*}
   \dtv (\pi^{(n)}, {\mathcal L}_j (X(r_n(\delta))) &\ge&  \pi^{(n)} (A) - \Pr_j (X(r_n(\delta)) \in A) \\
    &\ge& 1 - 4 \exp( - \frac{1}{32} \eps^{2} e^{4|\delta|}),
\end{eqnarray*}
uniformly in $ -\quarter\log n \le \d \le 0$, which is the required result.
\end{proof}

Our proof of the lower bound above is actually very similar to that of Diaconis and Shahshahani: we have obtained an improved result by using
Lemma~\ref{lemma1}, giving Gaussian concentration for $X(r_n(\delta))$, instead of appealing to Chebyshev's inequality.

We now turn to the upper bound.  We start by using Lemma~\ref{lemma1} to show that, for a long period beyond time $r_n(0) = \lfloor \frac14 n \log n \rfloor$, the process
$X(r)$ is unlikely to leave an interval of width $C \sqrt{n \log n}$ around $n/2$.

\begin{lemma}\label{lemma2}
For $n \ge 2e^4$, any $s \in \Z_+$, and any starting state~$j$,
$$
    \Pr_j \Bigl(\max_{0\le r\le s}|X(\adbng{r_n(0)+r}) - n/2| \ge \adbng{4}\sqrt{\tfrac n2 \log(\tfrac n2 )} \Bigr)
              \Le 16 (s+1)n^{-3}.
$$
\end{lemma}

\begin{proof}
For $t \ge r_n(0) = \lfloor \frac14 n \log n \rfloor$ and any starting state $j$, we have from~(\ref{ADB-new-mean}) that
\begin{equation*}  
    \Big | x_j(t) - \frac12 \Big | \Le \half \left( 1 - \frac2n \right)^t \Le \half e^{-\frac2n (\frac14 n \log n - 1)} = \half n^{-1/2} e^{2/n}
    \Le \frac34  n^{-1/2},
\end{equation*}
for all $n \ge 5$.

Therefore, at times $r_n (0)+r$, $r\ge0$, for any starting state~$j$ and for $n \ge 5$, we have
\eq\label{ADB-second-mean-bnd}
    |\E_j X(\adbng{r_n(0)+r}) - \tfrac n2 | \Le \tfrac34  n^{1/2}.
\en
Combining this with Lemma~\ref{lemma1}, we have
for $c \le 3 \sqrt{n}/4$,
$$
    \Pr_j (|X(\adbng{r_n(0)+r}) - n/2| \ge (c + 3/4) \sqrt{n}) \Le 2 e^{-c^2/2}, \quad \adbng{r\ge0}.
$$
We apply this inequality with $c = 4\sqrt{\half\log(n/2)}-3/4$, which is greater than $\sqrt{6 \log(n/2)}$ for $n > 2e^4$ (since $2\sqrt 2 -3/8 > \sqrt 6$), and deduce that
$$
    \Pr_j \Bigl(|X(\adbng{r_n(0)+r}) - n/2| \ge \adbng{4}\sqrt{\tfrac n2 \log(\tfrac n2 )} \Bigr)
              \Le 16 n^{-3}, \quad r\ge0.
$$
The required result now follows.
\end{proof}

We remark here that it would be relatively straightforward to complete the proof of cut-off at this point: we can exhibit a coupling
between two copies of the chain both remaining close to $n/2$, such that the distance between the two copies is stochastically dominated by a simple
lazy random walk -- such a proof would show quickly that the two copies coalesce by time $r_n(0) + \delta n$ with probability $1 - O(\delta^{-1/2})$.
(A similar argument is used by Eskenazis and Nestoridi~\cite{EN}, based on a discrete-time analogue of Proposition~\ref{hitting}.)
In order to establish the bound~(\ref{DS-2a}), we need a more precise argument.

For the moment we assume, for simplicity of exposition, that $n = 4k$ for some positive integer~$k$.
We consider the walk $Y=Y^{(n)}$ defined by $Y(r) = X(r_n(0)+r) - n/2 = X(r_n(0)+r) - 2k$, $r\ge0$,
which describes the evolution of~$X$ beyond the time~$r_n(0)$.  The transitions of this walk are given by:
\begin{eqnarray}
  p_{j,j+1} & := & \Pr[Y(r+1) = j+1 | Y(r) = j] \Eq \frac14 - \frac j{4k} + \Bigl (\frac j{4k}\Bigr )^2; \non \\
  p_{j,j-1} & := & \Pr [Y(r+1) = j-1 | Y(r) = j] \Eq \frac14 + \frac j{4k} + \Bigl(\frac j{4k}\Bigr)^2;
        \non \\
  p_{j,j} & := & \Pr[Y(r+1) = j | Y(r) = j] \Eq \frac12 - 2\Bigl(\frac j{4k}\Bigr)^2, \label{Y-trans}
\end{eqnarray}
for $-2k \le j \le 2k$.

At least when $j/4k$ is small, $Y$ has transition probabilities close to those of the simpler process $\tY := (\tY^{(n)}(r),\,r\ge0)$,
with $\tY(0) = Y(0)$, and transition probabilities given by
\begin{eqnarray}
  \tp_{j,j+1} & = & \Pr[\tY(r+1) = j+1 | \tY(r) = j] \Eq  \frac14 - \frac j{4k} ;\non \\
  \tp_{j,j-1} & = &  \Pr[\tY(r+1) = j-1 | \tY(r) = j] \Eq \frac14 + \frac j{4k} ;
         \label{tilde-Y-trans}\\
  \tp_{j,j} & = &  \Pr [\tY(r+1) = j | \tY(r) = j] \Eq  \frac12 .\non
\end{eqnarray}
We shall use~$\tY$ as a surrogate for~$Y$ in the argument to come.

The similarity of the transition probabilities \eqref{Y-trans} and~\eqref{tilde-Y-trans}, together
with Lemma~\ref{lemma2}, is next used to show that, with high probability,
the processes $Y$ and~$\tY$ are almost indistinguishable for a long time.

For a sequence $y := (y(r),\,r\ge0)$, we denote the initial segment up to time~$s$ by $y([0,s]) := (y(0),y(1),\ldots,y(s))$.

\begin{lemma} \label{lemma3}
For $n = 4k \ge 8000$, and $\displaystyle s \le \frac{k^2}{2500 \log^2 (2k)}$, we have
$$
    \dtv(\law(Y([0,s])),\law(\tY([0,s]))) \Le 100 s^{1/2}k^{-1}\log (2k).
$$
\end{lemma}

\begin{proof}
For a sequence $y := (y(r),\,r\ge0)$ such that the $y(r)$ are integers with $|y(r) - y(r-1)| \le 1$ for all~$r \ge 1$, let the {\em likelihood ratio} of the
process~$\tY$ compared to~$Y$ on the segment $y([0,s])$ be given by
\[
    \L(y([0,s])) \Def \prod_{r=1}^s \adbng{\frac{\tp_{y(r-1),y(r)}}{p_{y(r-1),y(r)}}}.
\]

For $k \ge 2000$, we set $\e_k =  5 \sqrt{2\log (2k)/k}$, and note that $\eps_k \le 1/2$.
If $|j|/k \le \e_k$, we then have, from the formulae for the transition probabilities, that
\eq\label{ADB-ratio-bnds}
  \max\Blb \Bigl| \frac{\tp_{j,j+1}}{p_{j,j+1}} - 1 \Bigr|,
    \Bigl| \frac{\tp_{j,j-1}}{p_{j,j-1}} - 1 \Bigr|,
     \Bigl| \frac{\tp_{j,j}}{p_{j,j}} - 1 \Bigr| \Brb \Le \half\e_k^2,
\en
so that, if $\L(y([0,s])) \le 2$, it follows that
\eq\label{ADB-quadratic-bnd-0}
     (\L(y([0,s+1])) - \L(y([0,s])))^2 \Le (\adbng{\half}\L(y([0,s]))\e_k^2)^2 \Le  \e_k^4.
\en

Replacing~$y$ by a path of~$Y$, we note that
$(\L(Y([0,s])),\,s\ge0)$ is a martingale.
Defining
\[
   \tau \Def \inf \Big \{s\ge0: \{\L(Y([0,s])) > 2\} \cup \{|Y(s)| > 5 \sqrt{2k\log (2k)}\} \Big \},
\]
it follows from~\Ref{ADB-quadratic-bnd-0} that the quadratic variation of the martingale
$\L(Y([0,r]))$ until time
$s \wedge \t$ is at most $s \e_k^4$. Since also $\ex\L(Y([0,s\wedge\t])) = 1$, it follows
from the Burkholder--Davis--Gundy inequality that
\eq\label{quad-variation}
  \ex \Big \{\big(\L(Y([0,s\wedge\t])) - 1\big)^2 \Big \} \Le  s \e_k^4.
\en

Define the events $A_s$ and $B_s$ by
\[
    A_s \Def \{\L(Y([0,s])) < 1\}; \quad B_s \Def \{\t > s\}.
\]
Then
\eqs
\lefteqn{\dtv(\law(Y([0,s])),\law(\tY([0,s]))) = \ex\{I[A_s](1- \L(Y([0,s])))\} \phantom{fyawevraywt}}\\
    &\le& \pr[\overline{B_s}] + \ex\{I[A_s \cap B_s](1- \L(Y([0,s])))\},
\ens
and, on~$B_s$, $s = s\wedge\t$.  Hence, 
\eqs
   \dtv(\law(Y([0,s])),\law(\tY([0,s]))) &\le & \pr[\overline{B_s}] + \ex\{(1- \L(Y([0,s\wedge\t])))_+\} \\
      &=& \pr[\overline{B_s}] + \frac12 \ex|1- \L(Y([0,s\wedge\t]))| \\
      &\le& \pr[\overline{B_s}] + \frac12 s^{1/2}\e_k^2.
\ens
From~\eqref{quad-variation} and Kolmogorov's inequality, and from Lemma~\ref{lemma2}, we have
\[
     \pr[\overline{B_s}] \Le s \e_k^4 + \tfrac1{4} (s+1)k^{-3} \le \tfrac54 s \e_k^4.
\]
Hence, for $s \le \e_k^{-4}$, we have
\begin{equation}\label{ADB-Y-to-Y-tilde}
    \dtv(\law(Y([0,s])),\law(\tY([0,s]))) \Le 2 s^{1/2}\e_k^2 
        = 100 s^{1/2}k^{-1}\log 2k,
\end{equation}
as required.
\end{proof}

Thus, with error at most $100 s^{1/2}k^{-1}\log 2k$, we can replace $Y([0,s])$ by~$\tY([0,s])$ when
calculating probabilities, and make only a small error if  $s \ll (k/\log k)^2$.
Recalling that $n=4k$, this means that the approximation of $Y$ by~$\tY$ is asymptotically accurate over time
intervals of length $o\bigl((n/\log n)^2\bigr)$.

We now use a coupling argument to show how fast~$\tY$ converges to its equilibrium
distribution~$\tipi\uk$.

\begin{lemma} \label{lemma4}
For any $k\ge 1$ and $r \ge 4k$, we have
$$
\dtv \Bigl(\law(\tY(r)),\tipi\uk\Bigr) \Le (k^{-1/2}\E|Y(0)| + 2) e^{-r/2k}.
$$
\end{lemma}

\begin{proof}
First, we note that the process~$\tY$ can equivalently be described by way of a discrete Ehrenfest ball scheme.
There are~$2k$ balls, each of which is in state $0$ or~$1$.  At each step, a ball is chosen independently at random from
the~$2k$ balls, and its state is chosen to be $0$ or~$1$, each with probability~$1/2$, independently
of the whole past of the process.  If~$k+j$ balls are in state~$1$ and $k-j$ in state~$0$
at step~$r$, we say that $\tY(r) = j$;  then the probabilities for~$\tY(r+1)$ are
easily seen to be given by~\eqref{tilde-Y-trans}, and its equilibrium distribution~$\tipi\uk$
to be $\Bi(2k,1/2)*\delta_{-k}$.

We now define a coupling of two copies $\tY^1$ and~$\tY^2$ of the process~$\tY$, with $\tY^1(0) \ge \tY^2(0)$.
Pair the balls in the two processes so that those initially in state~$1$ in~$\tY^2$ are paired with
balls in state~$1$ in~$\tY^1$, and  those initially in state~$0$ in~$\tY^1$ are paired with
balls in state~$0$ in~$\tY^2$; then pair the remaining $\tY^1(0) - \tY^2(0)$ balls in the two processes.
Couple the evolution by selecting one of these pairs of balls at each step, and re-assigning its state
independently (the new state being the same for both $\tY^1$ and $\tY^2$).
Let~$M(r)$ denote the number of pairs of balls that have not been
drawn up to step~$r$, made up of $M_{1}(r)$ in state~$1$, $M_{0}(r)$ in state~$0$,
and of~$M_{2}(r) = M(r) - M_{1}(r) - M_{0}(r)$ from the $\tY^1(0) - \tY^2(0)$
pairs of balls with differing initial states.  Conditional on $M_{0}(r)$, $M_{1}(r)$ and~$M_{2}(r)$, we have
\[
     \tY^1(r) = Z(r) + M_1(r) + M_2(r) - k \quad  \mbox{ and } \quad \tY^2(r) = Z(r) + M_1(r) -k,
\]
where $Z(r)$ has distribution
$\Bi(2k-M(r),1/2)$.  Now, since the distribution $\Bi(m,1/2)$ is unimodal with mode~$\lfloor m/2 \rfloor$,
we have, for all $m \ge 1$, that
\[
     \dtv(\Bi(m,1/2),\Bi(m,1/2)*\delta_1) \Eq \Bi(m,1/2)\{\lfloor m/2 \rfloor\} \ <\ \frac1{\sqrt m}.
\]
It follows that
\eqs
   \lefteqn{\dtv \Big ({\mathcal L} (\tY^1(r) | M_{0}(r),M_{1}(r),M_{2}(r)),
                  {\mathcal L} (\tY^2(r) | M_{0}(r),M_{1}(r),M_{2}(r)) \Big )}\\
        &&\Le \min\{1,M_{2}(r)(2k-M(r))^{-1/2}\}  \Le  M_{2}(r) k^{-1/2} + I[M(r) > k],
\ens
implying that
\eq\label{ADB-Y12-bnd}
   \dtv \Bigl(\law(\tY^1(r)),\law(\tY^2(r)) \Bigr)
        \Le k^{-1/2}\E M_{2}(r) + \pr[M(r) > k].
\en
Now $\pr(M(r) > k)$ is the probability that all the $k$ draws come from some subset of $k$ of the $2k$ matched pairs of balls,
and so, for $r \ge 4k$,
$$
\pr(M(r) > k) \le \binom{2k}{k} 2^{-r} \le 2^{2k-r} \le 2^{-r/2}.
$$
We also have $\ex M_{2}(r) = (\tY^1(0) - \tY^2(0))(1-1/(2k))^r \le (\tY^1(0) - \tY^2(0)) e^{-r/2k}$.
Hence, allowing either ordering of $\tY^1(0)$ and~$\tY^2(0)$, it follows from~\Ref{ADB-Y12-bnd}
that
\eq\label{ADB-Y12-bnd-2}
    \dtv \Bigl(\law(\tY^1(r)),\law(\tY^2(r)) \Bigr)
        \le k^{-1/2}|\tY^1(0) - \tY^2(0)| e^{-r/2k} + 2^{-r/2}.
\en
Setting $\tY^1(0) = Y(0)$, and taking~$\tY^2(0) \sim \tipi\uk$ to be in equilibrium,
we deduce, by taking expectations in~\Ref{ADB-Y12-bnd-2}, that
\begin{eqnarray*}
    \dtv \Bigl(\law(\tY^1(r)),\tipi\uk\Bigr) &\Le&
         \{k^{-1/2}(\adbng{\E}|Y(0)| + \sqrt{k/2})\} e^{-r/2k} + 2^{-r/2} \\
         &\le& (k^{-1/2} \E|Y(0)| + 2) e^{-r/2k},
\end{eqnarray*}
as desired.
\end{proof}

\begin{proof} [Proof of Theorem~\ref{thm.BL2}(b)]
We combine Lemma~\ref{lemma3} with Lemma~\ref{lemma4},
replacing $Y(r)$ by $X(r_n(0)+r) - n/2$, to deduce that, for any~$j$,
\eqa
    \lefteqn{\dtv \Bigl(\law(X(r_n(0)+r) - n/2),\tipi\uk\Bigr)} \non\\
     &\Le& ( k^{-1/2}\E_j |X(r_n(0)) - n/2| + 2 ) e^{-2r/n} + 100 r^{1/2}k^{-1}\log 2k \non\\
     &\Le&  4 e^{-2r/n} + 400 r^{1/2}n^{-1}\log n,  \label{ADB-dtv-binomial}
\ena
where we have used~\Ref{ADB-second-mean-bnd} to reach the last inequality,
provided $8000 \le 4k = n \le r \le n^2/40000\log^2(n/2)$.

The bound in~\Ref{ADB-dtv-binomial}
remains valid for {\it any\/} initial distribution; taking $X(0) \sim \pi\un$, so that also $X(r_n(0)) \sim \pi\un$,
this implies that
\[
   \dtv \Bigl(\pi\un * \delta_{-n/2},\tipi\uk\Bigr)
      \Le 4 e^{-2r/n} + 400 r^{1/2}n^{-1}\log n.
\]
also.  (The bound above is valid for any $r$, and is minimised for $r$ of order $n \log n$.  One could obtain a stronger bound, of order $n^{-1}$, by
direct computation, but this is rather delicate and the gain is not relevant to us.)

Hence, for $n$ a sufficiently large multiple of~4, and $n \le r \le \quarter n \log n - n \log \log n$, we have,
\eqa
   \dtv \Bigl(\law_n(X(r_n(0)+r)),\pi\un\Bigr) &\le&
      2\{ 4e^{-2r/n} + 400 r^{1/2}n^{-1}\log n\} \non\\
   &\le& 10 e^{-2r/n}. \label{ADB-DS-upper}
\ena
This bound also holds trivially for $r \le n$.
Taking $r = \d n$, this proves the result in the case where $n$ is a multiple of~4.

If~$n$ is not divisible by~$4$, the argument remains almost the same.
Define $k := \lfloor n/4 \rfloor$, and set $Y(r) := X(r_n(0)+r) - 2k$, as above.  The transition rates for~$Y$ are not quite as
in~\Ref{Y-trans}, but they are very close, resulting only in an extra contribution of order~$O(k^{-1})$
to the bounds in~\Ref{ADB-ratio-bnds}.
This correction is of smaller order than $\e_k^2$, and can be absorbed into the bound~\Ref{ADB-Y-to-Y-tilde} provided $k$ is sufficiently large.
The rest of the proof is unchanged.
\end{proof}

Diaconis and Shahshahani~\cite{ds87}, and other authors, actually consider a more general version, with boxes
of unequal sizes.  The first box initially contains~$n'$ red balls, and the second $2n-n'$
black balls.  The mixing process runs as before.  Our approach can be used for
this model as well.  The jump probabilities for the process counting the number~$X$ of red balls
in the first box are again quadratic in the current state~$j$ of the process.
When evaluated close to the equilibrium mean $n'p$, where $p := n'/2n$, these probabilities are close
to the linear jump probabilities near equilibrium of another process~$\tY$ consisting of~\grbc{$\ell$} balls, coloured
red or black, with the following dynamics.  At each time step, a ball is chosen.  It is left
with unchanged colour with probability $1-\th$; otherwise, it is re-coloured red with probability~$q$
and black with probability $1-q$, independently of everything else (so that its colour may in fact
still be unchanged). Then $\tY(r)$ denotes the number of red balls at time~$r$.
The values of $\ell,\th$ and~$q$ to best match the original process are found to be
$$
   q := 2p(1-p);\quad \th := \frac1{2(1 - 2p(1-p))} \quad \mbox{ and }\quad
   \ell := \left\lfloor \frac{np(1-p)}{1-2p(1-p)} \right\rfloor;
$$
note that, for $n'=n$, as previously, we have $p=1/2 = q$, $\th = 1$ and $\ell = \lfloor n/2 \rfloor$,
corresponding to the approximation made before.  With these modifications, an analogous argument can
be carried out, to establish cut-off.

\section{A two host model of  disease}\label{parasites}

Our next example is a two-dimensional Markov chain~$\wX\un$ in continuous time, representing
a two host model of disease, in which transmission only occurs between one host type
and the other (snails and human beings in schistosomiasis (Jordan, Webbe and Sturrock~\cite{JWS})), or males and females in
sexually transmitted diseases (Hethcote and Yorke~\cite{HY84})).  Our framework is appropriate for a disease that is not
naturally endemic in a region, being supported at a low level through immigration from outside.
In state $\bx := (x_1,x_2)^T \in \Z_+^2$, there are~$x_1$ type-$1$ hosts and~$x_2$ type-$2$ hosts infected.
From any state $\bx$, there are four possible transitions, whose rates are as follows:
\begin{eqnarray}
  (x_1,x_2) &\to& (x_1+1,x_2)\ \mbox{ at rate }\ \a x_2 +  \m n \non\\
  (x_1,x_2) &\to& (x_1,x_2+1)\ \mbox{ at rate }\ \b x_1 + \n n \non\\
  (x_1,x_2) &\to& (x_1-1,x_2)\ \mbox{ at rate }\ \g x_1  \non\\
  (x_1,x_2) &\to& (x_1,x_2-1)\ \mbox{ at rate }\ \d x_2. \label{ADB-new-rates}
\end{eqnarray}
Here, $\alpha$, $\beta$, $\gamma$, $\d$, $\m$ and~$\n$ are fixed \grbc{positive} constants, and the parameter~$n$ is a
measure of the typical size of the infected population.
The first transition corresponds to the infection of a type~$1$ host, by a type~$2$ host or from outside,
and the second to the infection of a type~$2$ host.
The third transition
corresponds to the recovery of a type~$1$ host, and the fourth to the recovery of a type~$2$ host.
The infection transition rates are appropriate in circumstances in which the host population is so large
that the reduction in infection rate caused by some of the population already being infected is negligible,
or for diseases such as malaria, when `super-infection' is possible: a host
infected more than once is proportionately more infectious -- in this case, $\bf x$ denotes the total number of infections of
each type of host.

Let $\bm(t) := \bm_\bx(t) := n^{-1}\E_\bx\{\wX\un(t)\}$, where $\E_\bx, \P_\bx$ and~$\law_\bx$ refer
to the distribution conditional on
$\wX\un(0) = \bx$.  It follows that~$\bm$ satisfies the differential equation $d\bm/dt = A\bm + \bb$,
where
\[
  A \Def \left(\begin{matrix}
                         - \g & \a \\ \b & - \d
                    \end{matrix} \right) \quad \mbox{and}
     \quad \bb \Def \left(\begin{matrix}
                         \m \\ \n
                    \end{matrix} \right),
\]
with initial condition $\bm(0) = n^{-1}\bx$.
\grbb{We define $R := \a\b/\g\d$, and assume from now on that $R < 1$, so that $A$ has both eigenvalues negative, and we denote them by
$\grbc{-\r > -\r'}$, with corresponding unit (right) eigenvectors \grbc{${\bf v}$ and ${\bf v}'$.}}
The differential equation has a non-trivial equilibrium at
\eq\label{ADB-detc-equilibrium}
   \bc \Def -A^{-1}\bb \Eq \frac{1}{\g\d(1-R)} \begin{pmatrix}\a\n + \d\m \\ \b\m + \g\n \end{pmatrix},
\en
and its full solution is
\eq\label{linear-solution}
     \bm_\bx(t) \Eq \bc + e^{At}(n^{-1}\bx - \bc),
\en
showing that the equilibrium~$\bc$ is globally attractive when $R < 1$.

For any $n$ and any $\bx \in \Z_+^2$, we define the travel time from state $\bx$ \grbc{(to within $n^{-1/2}$ of $\bc$)} to be
$$
t_n(\bx) \Def \inf\{t > 0\colon |e^{At}(n^{-1} \bx - \bc)| \le n^{-1/2}\},
$$
which, in view of \Ref{linear-solution}, is therefore the infimum of times $t$ such that $|\E_\bx\{\wX\un(t)\} - n \bc| \le n^{1/2}$.

For $0 < \z < 1$, let
$$
E_n(\z) := \{\bx \in \Z_+^2\colon n\z \le |\bx-n\bc| \le n/\z\}.
$$
We shall prove the following theorem.

\begin{theorem}\label{ADB-epidemic-thm}
Suppose that $R < 1$. Then, for any $0 < \z < 1$, $\wX\un$ exhibits cut-off at~$t_n(\bx)$
on~$E_n(\z)$, with window width~$1$.
\end{theorem}

We first consider the problem of estimating $t_n(\bx)$ for $\bx \in E_n(\z)$.
Writing $n^{-1} \bx - \bc$ as a linear combination $\lambda \bv + \lambda' \bv'$ of the unit eigenvectors $\bv$ and $\bv'$ of $A$, we have
$$
e^{At} (n^{-1}\bx - \bc) = \lambda e^{At} \bv + \lambda' e^{At} \bv' = \lambda e^{-\r t} \bv +\lambda' e^{-\r' t} \bv'.
$$
Then $t_n(\bx) \sim \max \{\r^{-1} \log (n^{1/2}\lambda), (\r')^{-1} \log (n^{1/2} \lambda')\}$.

For $\z \in (0,1)$, there is a constant $L_\z$ such that, for all $\bx \in E_n(\z)$, $t_n(\bx) \le \frac12 \r^{-1} \log n + L_\z$.
For ``most'' states in $E_n(\z)$, there is a matching lower bound, but $t_n(\bx)$ is as small as
$\frac12 (\r')^{-1} \log n + O(1)$ when $\frac1n \bx -\bc$ is close to a multiple of $\bv'$.

\medskip

The rest of this section is devoted to a proof of Theorem~\ref{ADB-epidemic-thm}: we give a brief road map of the proof here.
Our basic plan is to apply Theorem~\ref{ADB-contract-cts-lem} to our chain, showing concentration of measure for $\wX\un(t)$ while $t \le t_n(\bx)$.
To this end, we specify a suitable metric, and a Markovian coupling of two copies of the chain which is contracting in Wasserstein distance with respect to that metric.
We show that the chain remains within a good set (where, in particular, the total transition rate is bounded) over a long time period.  Then we apply
Theorem~\ref{ADB-contract-cts-lem} to each of the two coordinate projections, showing that both remain concentrated around their means for a long time.  We deduce readily
that the chain is far from its equilibrium for times less than $t_n(\bx)$.  On the other hand, once the chain reaches a neighbourhood of $n \bc$, we can use
Proposition~\ref{hitting} to show that it couples rapidly with an equilibrium copy of the chain, so the total variation distance to the equilibrium copy is small for times
only slightly greater than $t_n(\bx)$.

The two left eigenvectors of $A$ can be written in the form $(1,\xi)$, where $\xi$ is a solution of the
equation $\delta - \alpha / \xi = \gamma - \beta \xi$, with the common value $\delta - \alpha / \xi$
being minus the corresponding eigenvalue.  This equation has one negative solution $\xi=\theta'$, corresponding to the
eigenvalue $-\r'$, and the other solution $\xi = \theta$ lying in the interval $(\alpha/\delta, \gamma/\beta)$.
Thus we have
\begin{equation} \label{theta}
\delta - \frac\alpha\theta = \r = \gamma - \beta \theta.
\end{equation}

We introduce the norm $\| \cdot\|_\theta$ on $\R^2$, with
$$
\| \bx \|_\theta =: |x_1| + \theta |x_2|.
$$
We shall shortly prove that our chain has a contracting coupling with respect to the distance $\| \bx - \by \|_\theta$.

Next, we collect some elementary properties of the Markov chain~$\wX\un$.
First, we note that, for $R < 1$, $\wX\un$ is a $2$-type subcritical Markov
branching process with immigration, and hence has an equilibrium distribution~$\pi\un$.
Furthermore, since the process without immigration is sub-critical and has
birth and death rates that do not depend on~$n$, whereas the immigration rates
are multiples of~$n$, the mean of~$\pi\un$ is~$n\bc$, and its covariance
matrix is of the form $n\Sigma$, for~$\Sigma$ not depending on~$n$
(see, for example, Quine~\cite{Quine70} (Theorem on p.~414 and Equation~(29)) for analogues
in discrete time).

Next, for use with Theorem~\ref{ADB-contract-cts-lem}, we show that the chain rarely
gets too far from the origin, so that the total transition rate remains bounded.
For $H > 0$, we define
$$
D_n(H) := \{\bx\in\Z_+^2\colon \| \bx \|_\th \le H n\}.
$$

\begin{proposition}\label{deviations}
Suppose that $R < 1$.  Then there exist positive constants $C$ and $\ps$, depending on the parameters of the model but not on $n$, such that, for any
$H \ge 4 \| \bb \|_\th/\r$, \grbc{any $n \in \N$,} any $\bx \in D_n(H)$, and
any $T,w > 0$,
$$
\pr_\bx\Bigl[\sup_{0\le t\le T} \|\wX\un(t)\|_\theta > n(H + w) \Bigr] \Le C nT e^{-n\ps w}.
$$
\end{proposition}

\begin{proof}
Let $\cA\un$ denote the generator of~$\wX\un$, and define $h_{\ps}(\bx) := \exp\{\ps \|\bx\|_\theta\}$.
The first step is to show that, for sufficiently small positive~$\ps$,
$(\cA\un h_{\ps})(\bx) < 0$ for all~$\bx$ such that $\|\bx\|_\theta$ is large enough.

\grbb{
Setting $g(s) := s^{-2}(e^s-1-s)$ for $s\ne 0$, and $g(0)=1/2$, we have:
\begin{eqnarray} \label{drift}
(\cA\un h_{\ps})(\bx)
  &=& h_{\ps}(\bx)
   \bigl\{(\a x_2 + n\m)(e^\ps - 1) + \g x_1(e^{-\ps}-1) \nonumber \\
  &&\mbox{}\qquad + (\b x_1 + n\n)(e^{\th\ps}-1) + \d x_2(e^{-\th\ps}-1)\bigr\} \nonumber\\
  &=&  h_{\ps}(\bx) \psi \bigl\{ \a x_2 + n\m  - \g x_1 + \th \b x_1 +\th n \n - \th \d x_2 \bigr\}  \nonumber \\
  &&\mbox{}\qquad  + h_\ps(\bx) \ps^2 \bigl\{ (\a x_2 + n \m ) g(\ps) + \g x_1 g(-\ps) \\
  &&\mbox{} \qquad \qquad + \th^2 (\b x_1 + n \nu) g(\th\ps) + \th^2 \d x_2 g(-\th\ps) \bigr\}. \nonumber
\end{eqnarray}
We now see that
\eqs
\lefteqn{\a x_2 + n\m  - \g x_1 + \th \b x_1 +\th n \n - \th \d x_2} \\
&=& n(\m + \th\n) + (\a/\th - \d)x_2 \th + (\b\th - \g)x_1 \\
&=& n \| \bb \|_\th - \r x_2 \th - \r x_1 \\
&=& n \| \bb\|_\th - \r \| \bx \|_\th.
\ens
We bound the $\ps^2$ term in~\Ref{drift} above by noting that $g(\pm \ps)$ and $g(\pm \th \ps)$ are all at most~1,
provided $\ps \le 1/(1\vee \th)$, and hence
\eqs
\lefteqn{(\a x_2 + n \m ) g(\ps) + \g x_1 g(-\ps) + \th^2 (\b x_1 + n \nu) g(\th\ps) + \th^2 \d x_2 g(-\th\ps)} \\
&\le& \grbc{(\m + \th^2 \n) n + (\b \th^2 + \g) x_1 + (\a + \d \th^2) x_2} \\
&\le& (1 \vee \th) n \| \bb \| + (\a/\th + \b \th^2 + \g + \d \th) \| \bx \|_\th \phantom{fetwafrwea}.
\ens
Hence, for $\ps \le \min (1/(1\vee \th), \frac12\r/(\a/\th + \b \th^2 + \g + \d \th))$, we have
\begin{eqnarray} \label{second-drift}
(\cA\un h_{\ps})(\bx) &\le& h_\ps(\bx) \ps \Bigl[ n \| \bb \|_\th - \r \| \bx \|_\th \nonumber\\
&& \mbox{} \quad + \ps (1 \vee \th)  n \| \bb \|_\th + \ps (\a/\th + \b \th^2 + \g + \d \th) \| \bx \|_\th \Bigr] \nonumber\\
&\le &h_\ps(\bx) \ps \big[ 2 n \| \bb \|_\th - \r \| \bx \|_\th /2 \big],
\end{eqnarray}
which is non-positive whenever $\| \bx \|_\th \ge 4 n \| \bb \|_\th /\r$.
}

Now fix some $H \ge 4 \|\bb\|_\th /\r$, and some starting state $\bx \in D_n(H)$, so that $\| \bx \|_\th \le nH$
\grbc{and therefore $x_1 \le nH$ and $x_2 \le nH\th^{-1}$}.
Fix also some $w>0$.  We \grbc{will}
show that the probability that~$\wX\un$ ever exits the set $D_n(H+w)$ during a fixed time interval
$[0,T]$ is very small for large~$n$.


We consider the excursions out of the set $D_n(H)$ during $[0,T]$.
Note that, each time that $\wX\un$ enters~$D_n(H)$,
it remains there at least for the holding time of the state at which it first enters,
which has an exponential distribution with mean at least $1/n\grbc{q(H)}$, for
$$
    q(H) \Def \m + \n + \max\{\th^{-1}(\a +\d), (\b + \g)\}H.
$$
This implies that the number of exits of~$\wX\un$ from~$D_n(H)$ in~$[0,T]$ is stochastically
dominated by a Poisson random variable with mean~$nT q(H)$.

We claim that, each time that $\wX\un$ leaves~$D_n(H)$, the probability that~$\|\wX\un\|_\th$ exceeds the value~$n (H + w)$
before $\wX\un$ returns to~$D_n(H)$ is exponentially small in~$n$.
To prove this, consider starting in some state $\bf y$ which can be reached in one step from $D_n(H)$, so that $\| \by \|_\th \le nH + (1 \vee \th)$,
and let
\eqs
   \t_1 &:=& \inf\{t > 0\colon \wX\un(t) \in D_n(H)\};\\
   \t_2 &:=& \inf\{t > 0\colon \wX\un(t) \notin D_n(H + w)\}.
\ens
In view of~\Ref{second-drift}, $h_{\ps}(\wX\un(t\wedge\t_1))$ is a non-negative supermartingale
in $t \ge 0$.  Stopping at $\min\{\t_2,\t_1\}$, it thus follows that,
\[
  e^{n\ps H + \ps(1\vee \th)}  \ \ge \  h_{\ps}(\by) \ \ge\    e^{n\ps (w+H)} \pr_{\by}[\t_2 < \t_1],
\]
from which it follows that
\eq\label{ADB-small-prob1}
     \pr_{\by}[\t_2 < \t_1] \Le e^{-n\ps w} e^{\ps(1\vee \th)}.
\en

It follows that the expected number
of times that~$\wX\un$ exits~$D_n(H+w)$ in the interval $[0,T]$ is at most
$nT q(H) e^{\ps (1\vee \th)} e^{-n\ps w}$, establishing the proposition.
\end{proof}

We now introduce a Markovian coupling of two copies of the Markov chain $\wxn$, which we will then show to be contracting with
respect to the metric $d(\bx,\by) = \|\bx -\by\|_\th$ on $\Z_+^2$.  In this coupling, the two copies $U\un$ and $V\un$ make moves independently
in any co-ordinate where they currently differ (so in particular the two copies a.s.\ never move together in such a co-ordinate),
but make moves together as far as possible in co-ordinates where they currently agree.

For each ${\bf J} \in \cJ := \{(1,0)^T,(0,1)^T,(-1,0)^T,(0,-1)\}$,
we denote the transition rate of~$\wxn$ from $\bx$ to~$\bx+\bJ$,
given in~\Ref{ADB-new-rates}, by $r_{\bf J}(\bx)$.  We then
couple copies $U\un$ and~$V\un$ of~$\wxn$ as follows.

Suppose that $U\un(t) = \bu$ and $V\un(t) = \bv$.  If $u_1 \not= v_1$, then for ${\bf J} = (1,0)^T$ or $(-1,0)^T$, there is a transition to
$(\bu + {\bf J}, \bv)$ at rate $r_{\bf J}(\bu)$, and a transition to $(\bu, \bv + {\bf J})$ at rate $r_{\bf J}(\bv)$.  If $u_1 =  v_1$,
then there is a transition to $(\bu + {\bf J}, \bv + {\bf J})$ at rate $\min (r_{\bf J}(\bu), r_{\bf J}(\bv))$, a transition to
$(\bu + {\bf J}, \bv)$ at rate $\max (0, r_{\bf J}(\bu) - r_{\bf J}(\bv))$, and a transition to $(\bu, \bv + {\bf J})$ at rate
$\max (0, r_{\bf J}(\bv) - r_{\bf J}(\bu))$.  The transitions in directions $(0,1)^T$ and $(0,-1)^T$ are defined analogously.

\begin{proposition} \label{coupling}
The coupling defined above for $\wxn$ is contracting with respect to the metric $d(\bx,\by) = \| \bx - \by \|_\th$,
with constant $\r$.
\end{proposition}

\begin{proof}
If both chains make the same transition at~$t$, then the distance between them
does not change: $d(U\un(t),V\un(t)) = d(U\un(t-),V\un(t-))$.  Otherwise, the distance changes by
$\pm 1$ as a result of a jump by either copy in either $1$-direction, or by $\pm \th$ as a result of a jump
by either copy in either $2$-direction.

Let the generator of the process~$(U\un,V\un)$ be denoted by~$\hcA\un$.
We start by looking at the contribution of the $(-1,0)^T$ jumps to $(\hcA\un d)(\bu, \bv)$.
If $u_1 = v_1$, then $r_{(-1,0)^T}(\bu) = \g u_1 = \g v_1 = r_{(-1,0)^T}(\bv)$, so the two chains always make this
transition together, contributing no change to the distance.  If $u_1 > v_1$, then the $(-1,0)^T$ jump in $U\un$ occurs at
rate $\g u_1$ and reduces the distance by~1, while the $(-1,0)$ jump in $V\un$ occurs at
rate $\g v_1$ and increases the distance by~1: overall, the net contribution is $- \g |u_1 - v_1|$.  The same calculation applies if
$u_1 < v_1$, so in all cases the contribution of this jump is $-\g |u_1 -v_1|$.
Similarly, the contribution of the $(0,-1)^T$ jump is $-\d \th |u_2 - v_2|$.

We now turn to the $(1,0)^T$ jump.  If $u_1 = v_1$, the distance increases by~1 whenever one chain makes this jump and the other does not, which occurs
at rate $|r_{(1,0)^T}(\bu) - r_{(1,0)^T}(\bv)| = \a |u_2 - v_2|$.  If $u_1 \not= v_1$, a $(1,0)^T$ jump in one of the chains increases the
distance by~1, while the same jump in the other chain decreases the distance by~1, so the net contribution from this jump is at most
$|r_{(1,0)^T}(\bu) - r_{(1,0)^T}(\bv)|$, which is again equal to $\a |u_2 - v_2|$.  Similarly, the contribution of the
$(0,1)^T$ jump is at most $\b \th |u_1 -v_1|$.

Referring to \Ref{theta}, it follows that, for all states $\bu,\bv$,
\eq\label{ADB-new-contraction-este}
   (\hcA\un d)(\bu,\bv)
     \Le (-\g + \b\th)|u_1-v_1| + (-\d + \a/\th)\th|u_2-v_2| \grb{=} -\r d(\bu,\bv),
\en
as required.
\end{proof}

We will now apply Theorem~\ref{ADB-contract-cts-lem} to the Markov chain $\wxn$, with $f(\bx)$ either of the two
co-ordinate projections $f_1(\bx) = x_1$ or $f_2(\bx) = x_2$.
We fix some $0<\z <1$, and note that, for any $\bx \in E_n(\z)$, we have $|\bx - n \bc | \le n/\z$, and therefore
$$
\|\bx\|_\th \le \grbc{(1\vee \th)|\bx|} \le  (1\vee \th)(1/\z + |\bc|)n.
$$
Now we take $H = \max ((1\vee \th)(1/\z + |\bc|), 4 \|\bb\|_\th /\r)$, so that $E_n(\z) \subseteq D_n(H)$, and apply
Proposition~\ref {deviations} with $w = H$.
We see that, for any $\bx \in E_n(\z)$, and any $T>0$, the probability that the chain exits the set $D_n(2H)$ before time
$T$ is at most $CnT e^{-n \ps H}$, \grbc{for some constants $C$ and $\psi$}.  To apply Theorem~\ref{ADB-contract-cts-lem},
we take $\wS = D_n(2H)$, and note that, for $\by \in \wS$, the total
transition rate $-\wQ(\by,\by)$ out of state $\by$ is at most $q := n \bigl[\m +\n +2(\th^{-1}(\a + \d) +\b + \g)H\bigr]$.
If $f$ is the first co-ordinate projection~$f_1$, we have $|f_1(\bx) - f_1(\by)| \le \| \bx - \by \|_\th$, so we may take $L=1$: for $f=f_2$,
we need instead $L = 1/\th$.  We may also take $D = 1 \vee \th$.


Theorem~\ref{ADB-contract-cts-lem} now tells us that, for $i=1,2$, all $t >0$ and all $c > 0$, and all $\bx \in E_n(\z)$,
\eqs 
\lefteqn{   \P_\bx\Bigl( \bigl\{|\wxn_i(t) - \E_\bx\wxn_i(t)| > c\sqrt{n}\bigr\} \cap A_t \Bigr) }  \\
     &\Le& 2 \exp \Bigl( - \frac{c^2n}{n (\m + \n + 2(\frac{\a+\d}{\th} +\b+\g)H ) \frac{(1/\th \vee \th)^2}{\r} + \frac23 (1/\th \vee\th) c \sqrt n}\Bigr),
\ens
where
\[
     A_t \Def \Bigl\{\sup_{0\le s \le t} \|\wX\un(s)\|_\th  \le 2 n H\Bigr\}.
\]
Thus, for some constant $b$ depending on the parameters of the model and on $\z$, and all $c \le \eps \sqrt n$,
where $\eps > 0$ is sufficiently small, we have
\begin{equation}\label{ADB-new-conc}
\P_\bx\Bigl( \bigl\{|\wxn_i(t) - \E_\bx\wxn_i(t)| > c\sqrt{n}\bigr\} \cap A_t \Bigr) \le  2e^{-b c^2}
\end{equation}
for $i=1,2$, all $t > 0$ and all $\bx \in E_n(\z)$.

Moreover, for a suitable constant $K$, $t \le n$, and $c \le \eps \sqrt n$ for some sufficiently small $\eps > 0$,
\eq\label{ADB-LD-prob}
    \pr_{\bx}[\overline{A_t}] \Le C nt e^{-n\ps H} \le K e^{-b c^2},
\en
From~\Ref{ADB-new-conc} and~\Ref{ADB-LD-prob}, it now follows that, for $0 < t \le n$, $\bx \in E_n(\z)$, \grbc{and $c \le \eps \sqrt n$,}
\eq\label{ADB-new-conc-2}
   \P_\bx\Bigl( \bigl\{|\wxn(t) - n\bm_\bx(t)| > 2 c\sqrt{n}\bigr\} \Bigr)
     \Le  \grbc{(4 + K)} e^{-b c^2},
\en
for suitable constants $b$, $\eps$ and $K$, depending on the parameters of the model and on the choice of~$\z$.

We are now in a position to prove cut-off for our model.

\medskip

\begin{proofof}{Theorem~\ref{ADB-epidemic-thm}}
A lower bound on the mixing time can now easily be proved, much as in
the previous example, by considering the distribution of~$\wxn(t_n({\bf x}) - s)$, for $s > 0$.
Let $\grbb{\kappa} > 0$, depending on the parameters of the model, be such that
\eq\label{ADB-mean-growth}
   |e^{-As}\bz| \ \ge\ \kappa e^{\r s}|\bz|, \quad \mbox{ for all } \bz \in \R^2.
\en
By \grbb{(\ref{linear-solution}) and} the definition of $t_n (\cdot )$, we have
\eq\label{closeto}
n |m_{\bx}(t_n(\bx)) - \bc| = n |e^{At_n(\bx)} (n^{-1}\bx - \bc)| = n^{1/2}.
\en
Therefore, using~\Ref{ADB-mean-growth},
$$
n |m_{\bx}(t_n(\bx)-s) - \bc| =  n |e^{A(t_n(\bx)-s)} (n^{-1}\bx - \bc)| \ge \kappa n^{1/2} e^{\r s}.
$$

Let $B_s := \{\bw \in \Z_+^2\colon |\bw - n\bc| \le \half \kappa n^{1/2}e^{\r s}\}$.
Then, from \Ref{ADB-new-conc-2} with $c = \frac14 \kappa e^{\r s}$,
noting that $t_n(\bx) \le n$ for $\bx \in E_n(\z)$ provided $n$ is sufficiently large, we have
$$
\P_{ \bx }(\wxn(t_n(\bx) - s) \in B_s) \le (4+K) e^{-b \kappa^2e^{2\r s}/16}.
$$

On the other hand, as stated in the discussion before Proposition~\ref{deviations}, the covariance matrix of the equilibrium distribution of
$\wxn$ is of the form $n \Sigma$, with $\Sigma$ being independent of $n$. It hence follows, using Chebyshev's inequality, that
$\pi\un(B_s) \ge 1 - 4\kappa^{-2}ve^{-2\r s}$, with $v := \tr(\Sigma)$.

\grbb{This then gives, for a suitable constant~$K'$ and $s$ and $n$ sufficiently large,
\eq\label{ADB-new-conc-lower-bnd}
    \dtv(\law_\bx(\wxn(t_n(\bx)-s),\pi\un) \ \ge\ 1 - K' e^{-2\r s},
\en
for any $\bx \in E_n(\z)$.  This establishes the first part of the definition of cut-off in~\Ref{ADB-cutoff}.}

\medskip

We now turn to the upper bound.
We will apply Proposition~\ref{hitting} to the Markov chain $(U^{(n)},V^{(n)})$, where $U^{(n)}$ is a copy of the
started close to $n \bc$, $V^{(n)}$ is another copy in equilibrium, and the pair are coupled as in Proposition~\ref{coupling}.
We use the proposition to show that coalescence occurs quickly with high probability.

%
\grbb{Consider a copy $U^{(n)}$ of $\wxn$ starting from state $\bx$ and couple it with an equilibrium copy $V^{(n)}$,
as in Proposition~\ref{coupling}.}  For any fixed $\eps >0$, we choose $c\grbc{=c(\eps)}$ so that
$(4+K) e^{-bc^2} \le \eps/4$, and use (\ref{ADB-new-conc-2}) and (\ref{closeto}) to conclude that
$$
\Pr_{\bf x} (|U^{(n)}(t_n(\bx)) - n {\bf c}| > (c+1) n^{1/2} ) \le \eps/4,
$$
and similarly for the equilibrium copy $V^{(n)}(t_n(\bx))$.
Therefore, with probability at least $1-\eps/2$, we have
$$
\|U^{(n)}(t_n(\bx)) - V^{(n)}(t_n(\bx))\|_\th \le 2(c+1)(1\vee\th) n^{1/2}.
$$

We are now in a position to apply Proposition~\ref{hitting} to the function $\| U^{(n)}(t_n(\bx)+s) - V^{(n)}(t_n(\bx)+s)\|_\th$, for $s \ge 0$.
Condition~(i) of the proposition is satisfied by Proposition~\ref{coupling}, and condition~(ii) is satisfied with $B = 1 \vee \th$.
For condition~(iii), note that, if ${\bf u} \not= {\bf v}$, each of the chains
moves while the other does not -- and so the distance between the two chains changes by at least $1 \wedge \th$ -- at rate at least
$(\mu \wedge \nu)n$.  \grbd{Hence the generator of the quadratic variation process is at least $\sigma^2 := (1 \wedge \theta)^2(\mu \wedge \nu)n$ from all
states where coalescence has not occurred.}

Proposition~\ref{coupling} then implies that, on the event that
$\| U^{(n)}(t_n(\bx)) - V^{(n)}(t_n(\bx))\|_\th \le 2(c(\eps)+1)(1\vee\th) n^{1/2}$,
\grbd{the probability that coalescence has not occurred by time $s$ is at most
$$
\frac{4 (c(\eps)+1)(1\vee\th) n^{1/2}}{\sqrt s \sigma} = \frac{\phi(\eps)}{\sqrt s},
$$
where $\phi(\eps):= \frac{4(c(\eps)+1) (1\vee \theta)}{(1\wedge \theta) \sqrt{\mu \wedge \nu}}$.
For $s = s(\eps) = 4 \phi(\eps)^2/\eps^2$, we conclude that}
\begin{eqnarray*}
\lefteqn{\Pr (U^{(n)}(t_n(\bx)+ s(\eps))\not= V^{(n)}(t_n(\bx)+s(\eps)))} \\
&\le& \eps/2 + \Pr (\|U^{(n)}(t_n(\bx))-V^{(n)}(t_n(\bx)\|_\th) \le 2(c(\eps)+1)(1\vee\th) n^{1/2}) \le \eps.
\end{eqnarray*}
Since $V^{(n)}(t_n(\bx)+s(\eps))$ is in equilibrium, it follows that
$$
\dtv(\law_\bx(\wxn(t_n(\bx)+s(\eps))),\pi\un) \ \le\ \eps,
$$
as required for the second part of the definition of cut-off in~\Ref{ADB-cutoff}.
%
\end{proofof}

\section{Supermarket model}\label{supermarket}

In this section, we apply our general continuous-time inequality, Theorem~\ref{thm.concb-continuous}, to a range of instances of the
{\em supermarket model}.  This is a simple and natural model of a queuing system, introduced by Mitzenmacher~\cite{mitz} and
Vvedenskaya, Dobrushin and Karpelevich~\cite{VDK}, and studied extensively since;
see, for instance~\cite{LMcD2006} and \cite{BFL}, which contain other references to related literature.

The supermarket model (in continuous time) with parameters $(n,d,\lambda)$ ($n$ and $d$ natural numbers, $\lambda \in (0,1)$) is defined as follows.
There are $n$ servers, each with their own queue of customers, and customers arrive according to a Poisson process with rate $\lambda n$.
Each arriving customer inspects the queues for $d$ of the servers, chosen uniformly at random with replacement, and joins one of the shortest queues
among these~$d$; customers cannot subsequently switch to a different queue.  At each server, customer service times are iid exponential
of mean~1.

The memoryless property of the arrival and service processes means that the supermarket model can be viewed as a continuous-time jump Markov chain,
whose state space is the set $\Z_+^n$ of possible $n$-tuples of queue lengths.  The possible transitions are of two types:
(i)~departures, where each queue of positive length is shortened by one at rate $1$, and (ii)~arrivals, at total rate $\lambda n$, where some queue,
chosen by the procedure described above, is lengthened by~1.  To be precise, on an arrival, an ordered $d$-tuple of queues is chosen uniformly at
random from all the $n^d$ possibilities, and the first shortest queue in the list receives the arriving customer and is thus lengthened by~1.

Much of the initial interest in the supermarket model stemmed from its properties as a ``low-cost'' load-balancing mechanism: for $\lambda < 1$ a constant,
the maximum queue length in equilibrium is of order $\log n$ when $d=1$, but of order $\log \log n$ when $d$ is a constant at least~2.  In~\cite{BFL} and
this paper, we are interested in different ranges of parameters, where $\lambda$ tends to~1 from below as $n \to \infty$, while $d$ tends to infinity.
In these ranges, as shown in~\cite{BFL}, the load-balancing among the servers in equilibrium is close to perfect -- the maximum queue length is
a given constant $k$ with high probability, and most queues have length exactly~$k$ -- even though the system is nearly at full capacity.

For the rest of this section, as in~\cite{BFL}, we set $\lambda = 1 - n^{-\alpha}$, and $d = n^\beta$, where $\alpha$ and $\beta$ are
fixed constants in $(0,1)$.  We will assume throughout that
\begin{equation} \label{assume}
\beta < \alpha < 2\beta \mbox{ and } \alpha < (1+\beta)/2.
\end{equation}
For $\beta \le 1/3$, the corresponding range of $\alpha$ is thus $(\beta, 2\beta)$; for $1/3 \le \beta <1$, the corresponding range for $\alpha$ is
$(\beta, (1+\beta)/2)$.  Other parameter ranges come into the scope of~\cite{BFL} and, with a little more work, we could prove concentration results for
those too.

Theorem~6.1 from \cite{BFL} gives the general behaviour of the model in a variety of ranges, including this one (referring to that theorem,
assumptions~(\ref{assume}) are equivalent to setting $k=2$).  The basic result is that, in equilibrium, the chain lies in a ``good set'' where all queues
have length at most~2, with very high probability; it also states that, if the chain is started anywhere within an ``interior good set'', then with
high probability it remains in the good set for a long period of time.  We first set up notation, and then state the part of the result covering our range.

In fact, the model analysed in \cite{BFL} is a discrete-time variant of the continuous-time model studied here.  In that variant, the transition at each time-step is
an arrival with probability $\lambda/(1+\lambda)$ and a {\em potential departure} with probability $1/(1+\lambda)$.  If the transition is an arrival, a queue is chosen
as in the continuous-time version, and the length of that queue is increased by~1.  If the transition is a potential departure, a queue is chosen uniformly at random,
and the length of that queue is decreased by~1 if it is not empty.  If an empty queue is chosen for a departure, then the chain remains in its current state.
An alternative description of the continuous-time model is that events occur according to a Poisson process with rate $(1+\lambda)n$, and the transition associated with
an event is chosen as for the discrete-time model above.  A consequence is that the two models have the same equilibrium distribution, and
if the probability that the chain remains in some set $S$ of states for $k$ steps of the discrete chain is at least $q$, then the probability that the chain
remains in $S$ up to time $k/4n$ in the continuous model is at least $q$ minus the probability that a Poisson random variable with mean $k(1+\lambda)/4 \le k/2$ is
greater than~$k$, which is at least $q - e^{-k/8}$.  Similarly, provided $\lambda \ge 1/2$, if the total variation distance between the discrete-time supermarket model
after $k$ or more steps and the equilibrium distribution is at most $p$, then the total variation distance between the continuous-time model and the equilibrium
distribution is at most $p + e^{-k/16}$ for all times at least $k/n$.

For $n \in \N$, a state $x$ in $\Z_+^n$, and $j \in \N$, let $u_j(x)$ denote the proportion of queues in $x$ of length at least $j$.
Let $\eps =\eps(n)$ be any function such that $\eps \le 1/100$ and $\eps (n)^{-1} = o(n^\delta)$ for every $\delta > 0$.  For $n \in \N$, and
$\alpha$ and $\beta$ satisfying the inequalities in~(\ref{assume}), let $\cN^\eps(n,\alpha,\beta)$ be the set of states $x$ such that:
$$
\begin{array}{rcccl}
(1-6\eps)n^{-\alpha} &\le& 1 - u_1(x) &\le& (1+6\eps)n^{-\alpha} \\
(1-6\eps)n^{-\alpha+ \beta} &\le& 1 - u_2(x) &\le& (1+6\eps)n^{-\alpha+ \beta} \\
&&u_3(x) &=& 0
\end{array}
$$
A state $x$ in $\cN^\eps(n,\alpha,\beta)$ will thus have between $(1\pm 6\eps)n^{1-\alpha}$ empty queues, between $(1\pm 6\eps)n^{1-\alpha + \beta}$
queues of length~0 or~1 -- most of which will then have length~1 -- and the remaining queues all of length~2.  As $\beta < \alpha$, this implies that
the proportion of queues of length exactly~2 tends to~1 as $n \to \infty$.

The following result is taken from Theorems~6.1 and~1.2 of~\cite{BFL} -- in the application of Theorem~1.2, we take $t \ge n^2$ so that $t/3200n^{1+\beta} > \frac14 \log^2 n$
for $n$ sufficiently large; as is remarked after Theorem~10.5 of~\cite{BFL}, the conclusion is valid for the full range of $\eps$ stated above.  Note that the results
in~\cite{BFL} are stated for the discrete-time version of the model; we have derived results for the continuous-time version as described above, and bounded above the error
probabilities involved in the translation by $e^{-\frac14 \log^2 n}$.

\begin{theorem} [Brightwell, Fairthorne and Luczak] \label{thm:supermarket}
Given $n$ and $\eps(n)$ as above, and $\alpha$ and $\beta$ satisfying the inequalities in~(\ref{assume}), let $(Y(t))$ be a copy of the supermarket
process with parameters $(n,d,\lambda)$, where $\lambda = 1 - n^{-\alpha}$ and $d = n^\beta$, in equilibrium.  Then, for $n$ sufficiently large,
$$
\Pr\left( Y(t) \notin \cN^\eps (n,\alpha,\beta) \right) \le e^{-\frac14 \log^2 n}.
$$

Moreover, if $(X(t))$ is a copy of the supermarket process with $X(0) \in \cN^{\eps/6}(n,\alpha, \beta)$, then
$$
\Pr\left( X(t) \notin \cN^\eps (n,\alpha,\beta) \mbox{ for some } t \in [0,\frac{1}{4n} e^{\frac13 \log^2 n}] \right) \le 2 e^{-\frac14 \log^2 n},
$$
and, for $n$ sufficiently large and $t \ge n$,
$$
d_{TV} \big( \cL(X(t)), \Pi \big) \le 7n e^{-\frac14 \log^2 n},
$$
where $\Pi$ denotes the equilibrium distribution.
\end{theorem}

We will focus on the number $V(x) = n(1-u_1(x))$ of empty queues, and investigate how well $V(Y(t))$ is concentrated around its mean for an equilibrium
copy $(Y(t))$ of the supermarket process with parameters as above.  For $(Y(t))$, the mean total arrival
rate is $\lambda n = n(1-n^{-\alpha})$, while the mean total departure rate is the expected number of non-empty queues, which is
$n - \E V(Y(t)) = n \E u_1(Y(t))$.  In equilibrium, the mean arrival rate is equal to the mean departure rate, so we have
$\E V(Y(t)) = n^{1-\alpha}$.  States $x$ in $\cN^\eps(n,\alpha,\beta)$ thus all have $V(x)$ within $6\eps n^{1-\alpha}$ of the mean $\E V(Y(t))$.
We shall prove that
we have concentration of $V(Y(t))$ within $n^{(1-\beta)/2}$ of its mean $n^{1-\alpha}$: as $(1-\beta)/2 < 1-\alpha$, this is a sharper concentration result
than is given by Theorem~\ref{thm:supermarket}.
It is remarked in~\cite{BFL} that the proof of Theorem~\ref{thm:supermarket} goes through for $\eps = n^{-\delta}$, where $\delta$ is
sufficiently small: the implied result is still not as strong as we shall prove here, since $\delta$ would have to be strictly less than the minimum
of several quantities, one of which is $(1-\alpha) - (1-\beta)/2$, and this is the smallest of the quantities for part, but not all, of our range -- more details
can be found in the arXiv version of~\cite{BFL}.

The supermarket model is also used as an example by Luczak in~\cite{l08}, to illustrate the concentration inequality derived in that paper.
That analysis is based on a natural coupling $(X(t),Z(t))$ of two copies of the supermarket model with the same parameters, which we now describe -- our proof
is also based on this coupling.
In the coupling, the arrival times for the two processes are identical, and on an arrival the same ordered $d$-tuple of queues is inspected in the
two processes.  For each of the queues, a ``potential departure'' from the queue occurs at rate~1: for each of the copies of the process, if the
queue is non-empty at the time of the potential departure, a
customer is served and leaves the system at that time.  If states $x$ and $z$ are adjacent (i.e., one can be reached from the other by
a single transition), then they differ by~1 in exactly one queue.  For an adjacent pair $(x,z)$, we call the queue where the two states differ the
{\em unbalanced queue}, and we say that $x > z$ if the unbalanced queue is longer in $x$ than in $z$.  If $X(0) = x$ and $Z(0) = z$, where $x$ and $z$
are adjacent with $x > z$, then we claim that, under the coupling, the pair $(X(t),Z(t))$ remains adjacent, with $X(t) > Z(t)$, until the two copies
coalesce.  On a departure from the unbalanced queue, coalescence occurs if that queue is already empty in $Z(t)$, and otherwise the queue remains
unbalanced.  If an arriving customer joins the unbalanced queue in $x$, they join that queue in $z$ as well.  It is also possible that an arriving
customer joins the unbalanced queue in $z$ and a different queue in $x$; the states remain adjacent, but a different queue becomes unbalanced.

The analysis in~\cite{l08} assumes that $d$ is a constant, but it is easy to see that the proof there gives concentration around the mean only to
within order $\sqrt{nd}$.  For small enough $\beta$ and $\alpha$, this is still a stronger result than that implied by Theorem~\ref{thm:supermarket},
but the result we prove below always gives stronger concentration.

\begin{theorem} \label{concentration}
Let $(Y(t))$ be a copy of the supermarket model with parameters $(n,d,\lambda)$, in equilibrium, where $\lambda= 1 - n^{-\alpha}$, $d=n^\beta$,
and $(\alpha, \beta)$ satisfy~(\ref{assume}).
Then, for $n$ sufficiently large, and any $m$,
$$
\Pr \big( |V(Y(t)) -n^{1-\alpha}| > m \big) \le 2 \exp \left( - \frac{m^2 d}{112 n} \right) + \exp\left( -\frac15 \log^2 n\right).
$$
In particular, if $c=c(n)$ is positive, with $c = o(\log n)$, and $n$ is sufficiently large, we have
$$
\Pr \big( |V(Y(t)) -n^{1-\alpha}| > c n^{(1-\beta/2)} \big) \le 3 e^{- c^2/112}.
$$
\end{theorem}

Our proof will be an application of Theorem~\ref{thm.concb-continuous} to the (well-behaved) continuous-time chain $(Y(t))$.  We give the proof below,
postponing the proof of a key lemma.

\begin{proof}
We shall apply Theorem~\ref{thm.concb-continuous} to the supermarket model with the given parameters, with $f(x) = V(x)$, the number of empty queues
in state $x$.  We set $\eps = \eps(n) = 1/\log n$, and let $\widehat S$ be the set $\cN^\eps(n,\alpha,\beta)$.  We then consider starting in a state
$X(0) \in \cN^{\eps/6}(n,\alpha,\beta)$.
Note that, for any state $x$, the total
transition rate $q_x$ out of state $x$ is at most $2n$.  In order to apply the result, we need to identify a constant $\wbeta$ satisfying~(\ref{eq.coupling-cts}),
and a function $\weta (s)$ satisfying~(\ref{ADB-cts-key-bnd}).  We obtain these by analysing the natural coupling of two copies of the chain described above.

Accordingly, we consider a pair of copies $(X,Z)$ starting in adjacent states $x$ and~$z$ with $x > z$, evolving according to the coupling described above, so that
the two copies remain adjacent until coalescence.  At any time~$t$, $X(t)$ and $Z(t)$ are adjacent or equal, and if they are adjacent then there is one
unbalanced queue.  Let $L(t)$ denote the length of the longer unbalanced queue, or 0 if there is none, at time~$t$: the random process $(L(t))$ is thus a
function of the coupled pair $(X(t),Z(t))$, taking values in $\Z_+$, making steps up and down by~1, until it steps from~1 to~0 and remains at~0 thereafter.
For a pair $(x,z)$ of adjacent initial states, and $s \ge 0$, let $a_1(s) = a_1^{xz}(s)$ denote the
probability that $L(s)$ is equal to~$1$.

For an initial adjacent pair of states $(X(0),Z(0)) =(x,z)$ with $x > z$, and any time~$s$, the difference $V(X(s)) - V(Z(s))$ is equal to~1 when $L(s) = 1$ and~0 otherwise,
so the quantity $(\widehat P^s V)(x) - (\widehat P^s V)(z) = \E_x V(X(s)) - \E_z V(Z(s))$ is exactly equal to $a_1^{xz}(s)$.  In particular, we thus have
$|(\widehat P^s V)(x) - (\widehat P^s V)(z) | \le 1$, so we may take $\wbeta = 1$.

If $x \in \wS = \cN^\eps(n,\alpha,\beta)$, and $z$ is adjacent to $x$, then either $z \in \cN^{2\eps}(n,\alpha,\beta)$, or $z > x$ and $z$ has a queue of
length~3; in the latter case, the transition from $x$ to $z$ is an arrival in which only queues of length~2 are inspected, and the rate of such arrivals
from any state $x \in \cN^\eps(n,\alpha,\beta)$ is at most $(1-\frac12 n^{-\alpha+\beta})^d \le \exp (-\frac12 n^{-\alpha + 2\beta}) \le \exp(-\frac14 \log^2 n)$, for
$n$ sufficiently large.  As the total transition rate out of any state $x$ is at most $2n$, we have, a little crudely,
\begin{equation} \label{Q}
\sum_{z\in S} Q(x,z) \big( (\widehat P^s V)(x) -  (\widehat P^s V)(z) \big)^2 \le \exp(-\frac14 \log^2 n) + 2n \max_{z \sim x}
a_1^{xz}(s)^2,
\end{equation}
where the maximum is over initial pairs $(x,z)$ where $z$ is adjacent to $x$ and both are in $\cN^{2\eps}(n,\alpha,\beta)$.

\begin{lemma} \label{lem.bound}
\begin{equation} \label{rho}
a_1^{xz}(s) \le e^{-(d+2)s/2} + \frac{4}{d+2} e^{-s/(d+2)} + 2 e^{-\frac14 \log^2 n},
\end{equation}
whenever $x$ and $z$ are adjacent states in $\cN^{2\eps}(n,\alpha,\beta)$, and $s \le e^{\frac15 \log^2 n}$.
\end{lemma}

We postpone a proof of Lemma~\ref{lem.bound} until later, but we now indicate briefly what the terms in~(\ref{rho}) signify.  The final term accounts for the possibility
of leaving the set $\cN^{12\eps} (n,\alpha,\beta)$.  The first term accounts for the probability that $L(0)=1$ and no transition has occurred before time~$s$ to change the
length of the unbalanced queue.  The second term is the main term; roughly speaking, it arises from showing that coalescence occurs in time of order~$d$, and, conditional on
coalescence not occurring before time~$s$, the probability that $L(s) = 1$ is of order $1/d$.





We continue with the main thread of our proof, assuming the bound (\ref{rho}) in Lemma~\ref{lem.bound}.  Given this bound, we may take $\widehat{\alpha}(s)$
in~(\ref{ADB-cts-key-bnd}) to be
\begin{eqnarray*}
&& \exp( - \frac14 \log^2 n) + 2n \Big( e^{-(d+2)s/2} + \frac{4}{d+2} e^{-s/(d+2)} + 2 e^{-\frac14 \log^2 n} \Big)^2 \\
&\le& \exp( - \frac14 \log^2 n) + 6n \Big( e^{-(d+2)s} + \frac{16}{(d+2)^2} e^{-2s/(d+2)} + 4 e^{-\frac12 \log^2 n}\Big) \\
&\le& 2\exp( - \frac14 \log^2 n) + 6n e^{-(d+2)s} + \frac{96n}{(d+2)^2} e^{-2s/(d+2)},
\end{eqnarray*}
and
$$
\widehat \alpha_t = \int_0^t \widehat \alpha(s) \, ds
\le 2t e^{-\frac14 \log^2 n} + \frac{6n}{d+2} + \frac{48n}{d+2} +  \le \frac{55n}{d+2},
$$
for $t \le e^{\frac15 \log^2 n}$ and $n$ sufficiently large.

Now consider starting at any state $X(0) \in \cN^{\eps/6}(n,\alpha,\beta)$, and let $A_t$ be the event that the process stays within
$\widehat S = \cN^\eps(n,\alpha,\beta)$ until time~$t$.  For $t \le e^{\frac15 \log^2 n}$, Theorem~\ref{thm:supermarket} tells us that the
probability of $A_t^c$ is at most $2e^{-\frac14 \log^2 n}$.  We now apply
Theorem~\ref{thm.concb-continuous}, and obtain that, for any $m \ge 0$, and any $t \le e^{\frac15 \log^2 n}$,
\begin{eqnarray*}
\lefteqn{\Pr_{X(0)} \left( \left\{ |V(X(t)) - \widehat{P}^tV(X(0)) | > m \right\} \cap A_t \right)} \\
&\le& 2 \exp \left( -\frac{m^2}{110n/(d+2) + 2m/3} \right).
\end{eqnarray*}
For $m < n/d$, we have $\frac{m^2}{110n/(d+2) + 2m/3} \ge \frac{m^2d}{111n}$;
for $m \ge n/d$, we have that $\frac{m^2}{110n/(d+2) + 2m/3} \ge \frac{m}{111} \ge \frac14 \log^2 n$ for $n$
sufficiently large (depending on $\beta$).  Therefore we have, for any $m$ and $t \le e^{\frac15 \log^2 n}$,
\begin{eqnarray*}
\lefteqn{\Pr_{X(0)} \left( |V(X(t)) - \widehat{P}^tV(X(0)) | > m \right) } \\
&\le& 2 \exp \left( -\frac{m^2d}{111n} \right) + 4 \exp( - \frac14 \log^2 n ).
\end{eqnarray*}

The final part of Theorem~\ref{thm:supermarket} tells us that, for $t \ge n$ and any $X(0)$ in $\cN^{\eps/6}(n,\alpha,\beta)$, the total variation
distance between ${\mathcal L}_t^{X(0)}$ and the equilibrium distribution is at most $7n e^{-\frac14 \log^2 n}$.  Thus, choosing $t$ so that
$n \le t \le e^{\frac15 \log^2 n}$, we see firstly that $|\E V(X(t)) - \E V(Y(t))| \le 7n^2 e^{-\frac14 \log^2 n} \le 1$, for $n$ sufficiently large,
where $Y(t)$ is a copy in equilibrium.  Recalling that $\E V(Y(t)) =n^{1-\alpha}$, this yields that
\begin{eqnarray*}
\lefteqn{\Pr_{X(0)} \left( |V(X(t)) - n^{1-\alpha} | > m \right)} \\
&\le& 2 \exp \left( -\frac{(m-1)^2d}{111n} \right) + 4 \exp( - \frac14 \log^2 n )\\
&\le& 2 \exp \left( -\frac{m^2d}{112n} \right) + 4 \exp( - \frac14 \log^2 n ),
\end{eqnarray*}
for $m \ge 224$, and the inequality also holds trivially for $m < 224$ provided $n$ is sufficiently large.  We then further deduce that
$$
\Pr \left( |V(Y(t)) - n^{1-\alpha} | > m \right) \le 2 \exp \left( -\frac{m^2d}{112n} \right) + 8n^2 \exp( - \frac14 \log^2 n ),
$$
which implies the claimed result.
\end{proof}

It remains to prove Lemma~\ref{lem.bound}.  For this, we will use the following technical lemma, a variant of Gronwall's Lemma.

\begin{lemma} \label{gronwall-variant}
If $f(x)$ is continuous on $[0,\tau]$ and, for some $\gamma > 0$,
$$
f(s) \le f(t) - \gamma \int_t^s f(u) \, du \quad \mbox{ for all } 0 \le t \le s \le \tau,
$$
then $f(s) e^{\gamma s}$ is non-increasing on $[0, \tau]$ and so $f(s) \le f(0) e^{-\gamma s}$ for all $s\in [0,\tau]$.
\end{lemma}

\begin{proof}
Suppose for a contradiction that $f(s) e^{\gamma s} > f(t) e^{\gamma t}$, where $0 \le t < s \le \tau$.
Now take $t'$ to be the maximum value in $[t,s)$ such that $f(t') e^{\gamma t'} = f(t) e^{\gamma t}$.  By continuity, it follows that
$f(u) e^{\gamma u} \ge f(t') e^{\gamma t'}$ for all $u \in [t',s]$.

Applying the hypothesis to the times $t'$ and $s$, we obtain that
$$
f(s) \le f(t') - \gamma \int_{t'}^s f(u) \, du \le f(t') \left[ 1 - \gamma \int_{t'}^s e^{-\gamma(u-t')} \, du \right]
$$
$$
= f(t') e^{-\gamma (s-t')} = f(t) e^{-\gamma (s-t)},
$$
which gives the desired contradiction.
\end{proof}

\begin{proof} [Proof of Lemma~\ref{lem.bound}]
We need to show that, for $n$ sufficiently large, (\ref{rho}) holds whenever $x$ and $z$ are adjacent states in
$\cN^{2\eps}(n,\alpha,\beta)$, and $s \le e^{\frac15 \log^2 n}$.
Fix adjacent states $x$ and $z$ in $\cN^{2\eps}(n,\alpha,\beta)$ with $x > z$.  Let $\px$ be the probability that
either copy of the chain exits $\cN^{12\eps}(n,\alpha,\beta)$ before time $e^{\frac15 \log^2 n}$; by Theorem~\ref{thm:supermarket} (with
$\eps/6$ replaced by $2 \eps$), we have
\begin{equation} \label{pexit}
\px \le 2 e^{-\frac14 \log^2 n}.
\end{equation}

Until the copies coalesce, there is an unbalanced queue, with length $L(t)$ in $X(t)$ and length $L(t) -1$ in $Z(t)$; whatever the
length of the unbalanced queue, the rate of departures from the unbalanced queue is~1, and a departure would lead to coalescence
if $L(t)=1$, or reduce the unbalanced queue lengths by~1 if $L(t) \ge 2$.  If $L(t)=2$, then an arrival does not change $L(t)$ unless the process
leaves $\cN^{12\eps} (n,\alpha,\beta)$.  If $L(t)=1$, then an
arrival increases the length of the unbalanced queue exactly when the arriving customer joins the (empty) unbalanced queue in $Z(t)$.  The
rate $R_t$ of such arrivals depends on the number of empty queues in $Z(t)$; we could give an exact expression, but we content
ourselves with loose bounds that are easy to derive.  The rate $R_t$ is certainly at most the rate of
arrivals in which the unbalanced queue is inspected, which is equal to $\lambda n (1 - (1-1/n)^d) \le \lambda d \le d$.  Any arriving
customer who inspects the unbalanced queue and no other empty queue in $Z_t$ -- we call such an arrival a {\em critical arrival} -- will join
the unbalanced queue and thus cause $L(t)$ to increase from~1 to~2: while $Z(t)$ is in $\cN^{12\eps}(n,\alpha,\beta)$, the
proportion $p_t$ of empty queues in $Z(t)$ is at most $2 n^{-\alpha}$, and so the rate of critical arrivals is
$$
\lambda n \Big[ \big(1-p_t + \frac 1n\big)^d - (1-p_t)^d \Big] \ge \lambda n \frac dn (1-p_t)^{d-1} \ge \lambda d (1-p_t d)
$$
$$
\ge \frac34 d (1-2n^{-\alpha+\beta}) \ge \frac12 d,
$$
for $n$ sufficiently large.  Hence $R_t \ge d/2$ as long as $L(t) = 1$ and $Z(t)$ is in $\cN^{12\eps}(n,\alpha,\beta)$.

In summary, if $L(t)=1$, then $L$ decreases at rate~1, and increases at a rate $R_t$ between $d/2$ (provided $Z(t) \in \cN^{12\eps}(n,\alpha,\beta)$) and $d$.
If $L(t)=2$, then $L$ decreases at rate~1.

We consider the coupled pair of chains up to time $\tau = e^{\frac15 \log^2 n}$, starting from the initial state $(x,z)$.
Extending our earlier notation, we let $a_j(t) = \Pr(L(t) =j)$ and
$a_{\ge j}(t) = \Pr(L(t) \ge j)$, for $j=1,2$.

Applying Dynkin's formula, as well as the facts we have established about the rates of transitions for $L(t)$, we have that, for $t < s$,
\begin{equation} \label{eqa3}
a_{\ge 1}(s) = a_{\ge 1}(t) - \int_t^s a_1(u) \, du;
\end{equation}
\begin{equation} \label{eqa2}
a_{\ge 2}(s) = a_{\ge 2}(t) + \int_t^s \Big( - a_2(u) + \E [ 1_{L(u)=1} R_u ] \Big) du;
\end{equation}
\begin{equation} \label{eqa1}
a_1(s) = a_1(t) + \int_t^s \Big( - a_1(u) + a_2(u) - \E [ 1_{L(u)=1} R_u ] \Big) du.
\end{equation}
We also note that, for $u \le \tau$,
\begin{equation} \label{da1}
d a_1(u) \ge \E [ 1_{L(u)=1} R_u ] \ge \frac12 d (a_1(u) - \px).
\end{equation}

Recall that $\px$ is the probability that either copy leaves the set $\cN^{12\eps} (n,\alpha,\beta)$ before time $\tau$.  Note that
\begin{equation} \label{a2}
a_2(t) \le a_{\ge 2}(t) \le a_2(t) + \px
\end{equation}
for all $t \in [0,\tau]$.

Our aim is to prove the upper bound~(\ref{rho}) on $a_1(s) = a_1^{xz}(s)$ for all $s \le \tau$.
We shall establish that $a_{\ge 2}(s)$ is of order $d a_1(s)$, for $s$ larger than about $1/d$, and that $a_{\ge 1}(s)$ falls off at least
as fast as roughly $e^{-s/d}$.  This implies that the time to coalescence is approximately dominated by an exponential random variable with mean~$d$,
while, for $t$ greater than about $1/d$, conditional on coalescence not having occurred, the probability that $L(t) = 1$ is
of order $1/d$; these bounds  will yield~(\ref{rho}).
In our formal analysis, we shall use Lemma~\ref{gronwall-variant} several times.

We first consider the function
$$
r(s) = a_{\ge 2}(s) - (d+1) a_1(s) - \px.
$$
From (\ref{eqa2}), (\ref{eqa1}), (\ref{da1}) and (\ref{a2}), we have
\begin{eqnarray*}
r(s) &=& r(t) + \int_t^s \Big( (d+2) ( - a_2(u) + \E [ 1_{L_u=1} R_u ]) + (d+1) a_1(u) \Big) du \\
&\le & r(t) + \int_t^s \Big( (d+2) (-a_{\ge 2}(u) + \px + d a_1(u)) + (d+2) a_1(u) \Big) du \\
&= & r(t) - (d+2) \int_t^s r(u) \, du,
\end{eqnarray*}
and therefore from Lemma~\ref{gronwall-variant} we have that $r(s) \le r(0) e^{-(d+2)s} \le e^{-(d+2)s}$.  Rearranging, we obtain that, for $s \le \tau$,
\begin{equation} \label{eqa1s}
a_1(s) \ge \frac{1}{d+2} \left[ a_{\ge 1}(s) - e^{-(d+2)s} - \px \right].
\end{equation}
This tells us that, roughly speaking, after a lead-in time of order $1/d$, the probability $a_1(s)$ that the unbalanced queue has length~1
is at least about $1/(d+2)$ times the probability $a_{\ge 1}(s)$ that coalescence has not occurred.

The next step is to use the above to show that $a_{\ge 1}(s)$ falls off at least
as fast as roughly $e^{-s/d}$.  We see from (\ref{eqa3}) and (\ref{eqa1s}) that
$$
a_{\ge 1}(s) \le a_{\ge 1}(t) - \frac{1}{d+2} \int_t^s \Big( a_{\ge 1}(u) - e^{-(d+2)u} - \px \Big) du
$$
Now we consider the function
$$
v(s) = a_{\ge 1}(s) + \frac{2}{(d+2)^2} e^{-(d+2)s} - \px.
$$
We have
\begin{eqnarray*}
\lefteqn{\int_t^s \Big( a_{\ge 1}(u) - e^{-(d+2)u} - \px \Big) du} \\
&=& \int_t^s \Big( v(u) - \Big(1 + \frac{2}{(d+2)^2}\Big) e^{-(d+2)u} \Big) du \\
&\ge& \int_t^s v(u) \, du - \frac{2}{d+2} (e^{-(d+2)t} - e^{-(d+2)s}),
\end{eqnarray*}
and so
\begin{eqnarray*}
v(s) &\le& v(t) + \frac{2}{(d+2)^2} (e^{-(d+2)s} - e^{-(d+2)t}) - \frac{1}{d+2} \int_t^s v(u) \, du \\
&&\mbox{} + \frac{2}{(d+2)^2} (e^{-(d+2)t} - e^{-(d+2)s}) \\
&=& v(t) - \frac{1}{d+2} \int_t^s v(u) \, du.
\end{eqnarray*}
Therefore, by Lemma~\ref{gronwall-variant},  $v(s) \le v(0) e^{-s/(d+2)} \le 2 e^{-s/(d+2)}$, and we deduce that, for $s \le \tau$,
\begin{equation} \label{rft}
a_1(s) + a_{\ge 2}(s) = a_{\ge 1}(s) \le 2 e^{-s/(d+2)} + \px.
\end{equation}

Finally, we show that $a_1(s)$ is at most about $2/d$ times $a_{\ge 1}(s)$.
We apply Lemma~\ref{gronwall-variant} to the function
$$
q(s) = \frac d2 a_1(s) - a_{\ge 2}(s) - \frac d2 \px.
$$
From (\ref{eqa2}), (\ref{eqa1}) and (\ref{da1}), we have, for $t< s$,
\begin{eqnarray*}
q(s) &=& q(t) + \int_t^s \Big( -\frac d2 a_1(u) + \Big(\frac d2 +1\Big) \big(a_2(u) - \E [ 1_{L_u=1} R_u ]\big) \Big) du \\
&\le& q(t) + \Big(\frac d2 +1\Big) \int_t^s \Big( a_{\ge 2}(u) - \frac d2 (a_1(u) - \px) \Big) du \\
&\le& q(t) - \Big(\frac d2 +1\Big) \int_t^s q(u) \, du.
\end{eqnarray*}
We obtain that $q(s) \le q(0) e^{-(d/2 + 1)s} \le \frac d2 e^{-(d/2+1)s}$, so
$$
\frac d2 a_1(s) - a_{\ge 2}(s) \le \frac d2 e^{-(d/2+1)s} + \frac d2 \px.
$$

Summing with (\ref{rft}) yields, for $s \le \tau$,
$$
\Big(\frac d2 + 1\Big) a_1(s) \le \frac d2 e^{-(d/2+1)s} + 2 e^{-s/(d+2)} + \Big(\frac d2 +1\Big) \px,
$$
and so
$$
a_1(s) \le e^{-(d/2+1)s} + \frac{4}{d+2} e^{-s/(d+2)} + \px,
$$
which is the required bound.
\end{proof}

Theorem~\ref{concentration} gives concentration of the random variable $V(Y(t))$ about its mean within order $\sqrt{n/d}$.
We note that no such bound can be shown if we rely only on the fact that $V(x)$ is a Lipschitz function of the state space.  Indeed, coalescence of the
Markov chain takes time of order $d$, and the results of~\cite{l08}, \cite{Ollivier} or~\cite{Paulin} would only give concentration within order
$\sqrt {nd}$ of the mean.

We indicate briefly why we expect that concentration of $V(Y(t))$ within order $\sqrt{n/d}$ of its expectation is best possible.  If we look at the
transitions of the process over a time period $[0,t]$ of length $t=nd$, the number of arrivals has fluctuations of order $\sqrt {nd}$.
The analysis in the proof of Theorem~\ref{concentration} and Lemma~\ref{lem.bound} suggests that a positive proportion of the extra customers will still be in the
system at the end of the period, and approximately a proportion $1/d$ of these will be in queues of length~1, so that
fluctuations of order $\sqrt{nd}$ in the number of arrivals during $[0,t]$ result in fluctuations of order $\sqrt{n/d}$ in the number of empty queues at time~$t$.


We believe that a similar proof can be used to show sharp concentration of measure results for the supermarket model in the range where $\lambda < 1$ and $d \ge 2$
are fixed constants.  Here it is known that the proportion of queues of length at least $k$, for each $k$ fixed, is close to
$v(k) = \lambda^{(d^k-1)/(d-1)}$ in equilibrium.  For $k \ge 1$, let $f_k(x)$ be the number of queues at least $k$; for $x$ any state with approximately
$nv(k)$ queues of length $k$ for each $k$, and $k$ large, the quantity $Q(x,y) ((\widehat{P}^s f_k) (x) - (\widehat{P}^s f_k) (y))^2$ is dominated by
terms where the transition from $x$ to $y$ creates an unbalanced queue of length~$k$, and there is no departure from the unbalanced queue before
time $s$.  Thus we may take $\widehat{\alpha}(s)$ at most some constant times $n v(k) e^{-2s}$, and obtain concentration within order $\sqrt{n v(k)}$
for $f_k(x)$ in equilibrium, at least for $k$ large.

\medskip
{\bf Acknowledgements.}  MJL thanks Monash University for their
kind hospitality while part of this work was accomplished.  GRB thanks the University of Melbourne for their
equally kind hospitality while a different part of the work was accomplished.

\end{document}